\documentclass[11pt, reqno]{amsart}
\usepackage{subfiles}
\usepackage{standalone}
\usepackage{bm}
\usepackage[hidelinks]{hyperref}
\usepackage{multirow}
\usepackage{tikz}
\usepackage{caption}
\usepackage{textgreek}
\usepackage{float}
\usetikzlibrary{
 decorations.pathmorphing,%
 decorations.markings,%
 decorations.pathreplacing,%
 decorations.text,%
 calc,%
 patterns,%
 shapes.geometric,%
 arrows,
 decorations.shapes,
 positioning,
 plotmarks,
 shadings,
 math,
 intersections
}
\usetikzlibrary{arrows.meta,decorations.markings,angles,quotes}

\definecolor{jeffColor}{RGB}{102, 0, 204}
\definecolor{yaizaColor}{RGB}{0, 153, 153}
\definecolor{certainty}{RGB}{64, 228, 198}
\definecolor{hope}{RGB}{228, 194, 64}

\definecolor{roColor}{RGB}{183, 40, 46}
\definecolor{roColorBlue}{RGB}{40, 148, 200}
\definecolor{roColorGreen}{RGB}{0, 123, 67}

\definecolor{periodColor}{RGB}{255, 167, 105}
\definecolor{dark-green}{RGB}{135, 194, 130}
\tikzset{>=latex} % sets default arrow type in tikz
\tikzset{font=\small}
\tikzset{mark size=1.5pt, mark options=thin}
\tikzset{pin distance=4pt,
 every pin edge/.style={<-, thin, shorten <= -2pt}}
\usepackage{array,booktabs,tabularx}
\usepackage{graphicx}
\usepackage{amssymb}
\usepackage{arydshln}
\usepackage{enumerate}
\usepackage{color}
\usepackage{esint}
\usepackage{mathtools}
\usepackage{dsfont}
\usepackage{ esint }
\usepackage{placeins}
\usepackage[show]{ed}
\usepackage[final]{showkeys}
\usepackage{mathrsfs}
\usepackage{changepage}   % for the adjustwidth environment
\usepackage{comment}
\usepackage{enumitem}
\usepackage{soul}
\usepackage{needspace}
\usepackage[normalem]{ulem}

\usepackage{multicol}
\usepackage{tabto}
\usepackage{hyperref}
\usepackage[makeroom]{cancel} %added by Marcello for reviewing purpose

\usepackage{bbm}

\usetikzlibrary{shapes.geometric, arrows}
\tikzstyle{box} = [rectangle, 
minimum width=2.75cm, 
minimum height=1cm, 
text centered, 
text width=2.75cm, 
draw=black]
\tikzstyle{impbox} = [rectangle, double,
minimum width=2.75cm, 
minimum height=1cm, 
text centered, 
text width=2.75cm, 
draw=black]

\definecolor{uipoppy}{RGB}{221,128,71}
\definecolor{uipaleblue2}{RGB}{179,196,215}
\definecolor{uiviolet}{RGB}{86,86,99}
\definecolor{uiblack}{RGB}{0, 0, 0}
\definecolor{azul}{RGB}{0,128,255}
\definecolor{verde}{RGB}{50,180,50}
\definecolor{pale-verde}{RGB}{155,207,145}
\definecolor{uipaleblue}{RGB}{108,199,220}
\definecolor{light-gray}{RGB}{198,198,198}

\newtheorem{lemma}{Lemma}
\newtheorem{theorem}{Theorem}
\numberwithin{theorem}{section}

\newtheorem{corollary}[lemma]{Corollary}

\theoremstyle{definition}

\newtheorem{remark}[lemma]{Remark}

\usepackage{dsfont}

\newcommand{\mc}[1]{\mathcal{#1}}

\usepackage{multicol}
\usepackage{enumitem}

\def\XXint#1#2#3{{\setbox0=\hbox{$#1{#2#3}{\int}$} \vcenter{\hbox{$#2#3$}}\kern-.5\wd0}}

\DeclareMathOperator{\vol}{vol}

\DeclareMathOperator{\supp}{supp}

\newcommand{\e}{\epsilon}

\numberwithin{equation}{section}
\numberwithin{lemma}{section}

\usepackage{subcaption}

\newcommand{\tr}{\operatorname{tr}}

\newcommand{\sgn}{\operatorname{sgn}}

\newcommand{\edra}{\begin{color}{roColor}}
\newcommand{\edrb}{\end{color}}
\usepackage{marginnote}

\title[]{Lower bounds on Steklov eigenfunctions for a class of curvilinear polygons}

\author{Jeffrey Galkowski}
\address{Department of Mathematics, University College London, London, UK}
\email{j.galkowski@ucl.ac.uk}

\author{Ruoyu P. T. Wang}
\address{Department of Mathematics, Yale University, New Haven, CT, USA}
\email{ruoyu.wang@yale.edu}

\begin{document}

\maketitle

\begin{abstract}
Steklov eigenfunctions, $\{\phi_{\sigma_j}\}_{j=0}^\infty$, are the eigenfunctions of the Dirichlet to Neumann map on a bounded domain with eigenvalue $\sigma_j$. This paper considers Steklov eigenfunctions and eigenvalues on a family of curvilinear polygons, $\Omega_\alpha$ whose boundaries are smooth away from a single corner of angle $\alpha$. We (i) prove optimal lower bounds on the decay of the harmonic extension, $u_{\sigma_j}$, of $\phi_{\sigma_j}$ away from the boundary of $\Omega_\alpha$ and (ii) give optimal lower bounds on the degree to which the high energy Steklov spectrum of $\Omega_\alpha\setminus H_1$ may differ from that of $\Omega_\alpha\setminus H_0$ when $H_1$ and $H_0$ are isometric. This shows that the corresponding estimates in the recent work by the authors together with Malagutti are optimal~\cite{GaMaWa:26}.

The proof in this paper was obtained through an interaction with ChatGPT 6 Astra: While writing the work~\cite{GaMaWa:26}, the authors became curious about optimal error bounds for the difference between the Steklov spectra on a curvilinear polygon, $\mathcal{P}$, and the spectrum of \emph{any} operator which knows only local information about $\partial\mathcal{P}$. The authors suggested to ChatGPT to remove two isometric, smooth domains from a curvilinear polygon. ChatGPT suggested the ideas for a proof based on a regularized trace and the singularity in the kernel of the DtN map near a corner. The authors then suggested a simpler, conformal approach, the method for handling small angles, and the idea to obtain lower bounds on Steklov eigenfunctions. The paper is entirely written by the authors and the argument has been significantly modified, clarified, and simplified from the ChatGPT ideas.

% The proof in this paper was obtained through an interaction with ChatGPT 6 Astra: While writing the work~\cite{GaMaWa:26}, the authors became curious about optimal error bounds for the difference between the Steklov spectra on a curvilinear polygon, $\mathcal{P}$, and the spectrum of \emph{any} operator which knows only local information about $\partial\mathcal{P}$. The authors suggested to ChatGPT to remove two isometric, smooth domains from a curvilinear polygon and obtain a lower bound on the difference between the spectra of the resulting domains. ChatGPT suggested the ideas for a proof for reentrant corners based on a regularized trace and the singularity in the kernel of the Dirichlet to Neumann map near a corner using much more complicated methods than those appearing in the text. The authors then suggested a simpler, conformal approach, the method for handling small angles, and the idea to obtain lower bounds on Steklov eigenfunctions. The paper is entirely written by the authors and the argument has been significantly modified, clarified, and simplified from the ChatGPT ideas.
\end{abstract}

\section{Introduction}

Steklov eigenfunctions for $\Omega\Subset \mathbb{R}^d$ with Lipschitz boundary are non-zero  $\phi_{\sigma_j}\in L^2(\partial\Omega)$ solving
$$
-\Delta u_{\sigma_j}=0\text{ in }\Omega,\qquad \partial_{\nu}u_{\sigma_j}-\sigma_j u_{\sigma_j}=0\text{ on }\partial\Omega,\qquad u_{\sigma_j}|_{\partial\Omega}=\phi_{\sigma_j}.
$$
That is, $\phi_{\sigma_j}$ are the eigenfunctions with eigenvalue $\sigma_j$ for the Dirichlet-to-Neumann map, $\Lambda^\Omega$ satisfying $\Lambda^\Omega f:=\partial_\nu \mc{H}^\Omega f$, where $\mc{H}^{\Omega}:H^{\frac{1}{2}}(\partial\Omega)\to H^1(\Omega)$ satisfies
$$
-\Delta \mc{H}^\Omega f=0\text{ in }\Omega,\qquad  \mc{H}^\Omega f|_{\partial\Omega}=f.
$$
It is well known that there are $0=\sigma_0<\sigma_1\leq \sigma_2\dots$ such that $\{\phi_{\sigma_j}\}_{j=0}^\infty$ form an orthonormal basis for $L^2(\partial\Omega)$~\cite{ArMa:12}. 

This paper considers a family of domains in $\mathbb{R}^2\cong \mathbb{C}$ given as follows. Let $\gamma:\mathbb{R}/2\pi \mathbb{Z}\to \Gamma\subset \mathbb{C}$ be a smooth diffeomorphsim $0\notin \supp(\gamma(s)-s)$ and such that 
$$
\Gamma \subset \{ z\in\mathbb{C}\,:\, 0\leq \arg z<2\pi  \min(\tfrac{\pi}{\alpha},1)\}.
$$
Let $\mathbb{C}\setminus \Gamma=:\Omega\sqcup \Omega_{\rm out}$, where $\Omega $ is bounded, define 
\begin{gather*}
f_{\alpha}:\{z\in\mathbb{C}\,:\, 0\leq \arg z<2\pi  \min(\tfrac{\pi}{\alpha},1)\}\to \mathbb{C},\\
f_{\alpha}(z):= \begin{cases}\exp\big(\frac{\alpha}{\pi}\log z\big)&z\neq 0\\ 0&z=0\end{cases},\qquad \text{with the branch $\log$ chosen so that $0\leq \Im \log z   <2\pi$ on $\Omega$.}
\end{gather*}
and set $\Omega_{\alpha}:=f_{\alpha}(\Omega)$. We denote by
$$
\mc{E}:=\big\{\tfrac{\pi}{n}\,:\, 0<n\in\mathbb{Z}\big\},
$$
the set of \emph{exceptional angles}.

Our first theorem provides a lower bound on $u_{\sigma_j}$ in $\Omega_{\alpha}$. 
\begin{theorem}\label{t:lowerBound}
     Let $\alpha\in (0,2\pi)\setminus \mc{E}$. Then for all $x_0\in \Omega_{\alpha}$ and $r>0$ ,
    $$
    \limsup_{j\to \infty}\sigma_j^{\frac{\pi}{\alpha}}\|u_{\sigma_j}\|_{L^2(B(x_0,r)\cap \Omega_{\alpha})}>0.
    $$
\end{theorem}
When a domain $\Omega$ has smooth boundary $u_{\sigma_j}=O(\sigma_j^{-\infty})$ in the interior of $\Omega$ ~\cite{HLu,GalkToth,GPPS} and this, perhaps surprisingly, persists for polygons all of whose angles lie in $\mc{E}$~\cite{LPPS2022}.  Theorem~\ref{t:lowerBound} is the first result demonstrating that Steklov eigenfunctions need not decay superpolynomially into the interior. In another very recent work~\cite{DeLaPa:26}, the authors show that, on domains with $C^{2}$ boundary, harmonic Steklov eigenfunctions need not decay superpolynomially. Moreover, in a recent paper~\cite{GaMaWa:26}, the authors together with Malagutti show that, at least modulo an $\e$ loss in the power, Theorem~\ref{t:lowerBound} is optimal in the sense that on a curvilinear polygon all of whose angles are less than or equal to $\alpha_{\max}$ every eigenfunction decays like $C_\e\sigma_j^{-\frac{\pi}{\alpha_{\max}}+\e}$ away from the boundary. 

\begin{figure}
    \centering
   
    \begin{minipage}{.45\textwidth}
    \centering
  \includegraphics[width=\textwidth]{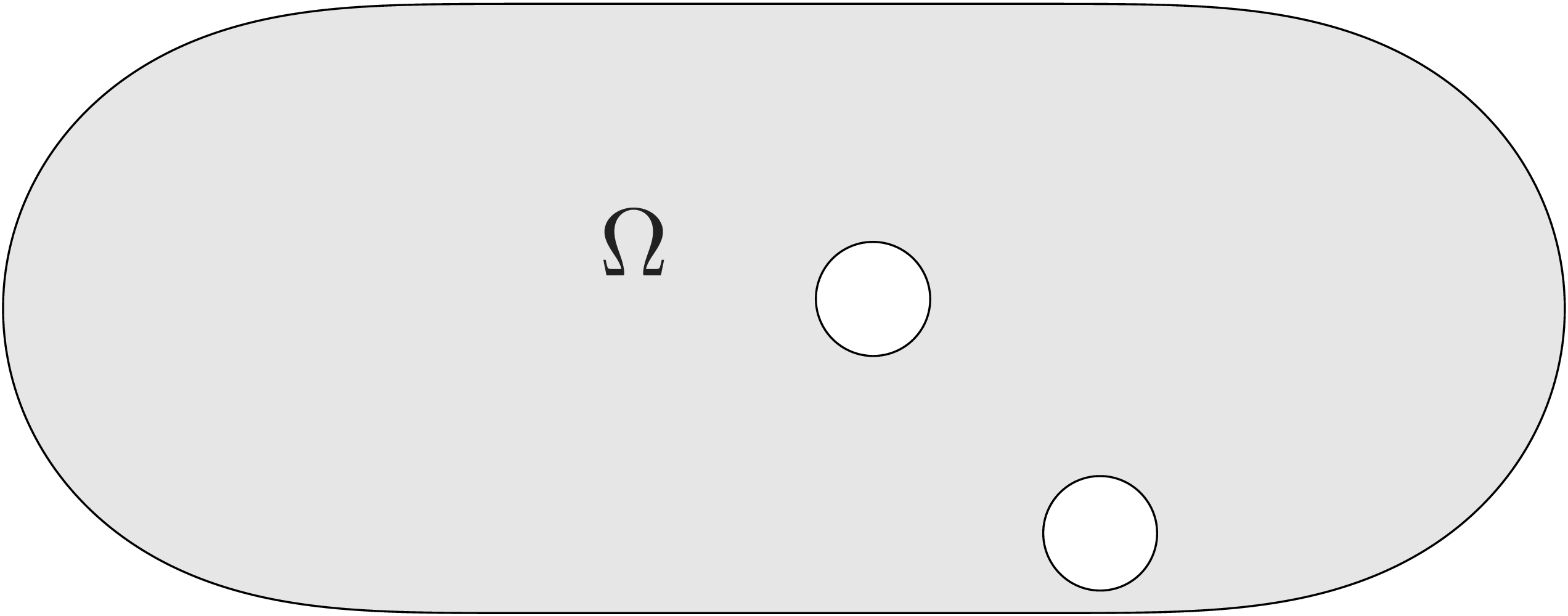}
  \end{minipage}
     \begin{minipage}{.45\textwidth}
    \centering
  \includegraphics[width=\textwidth]{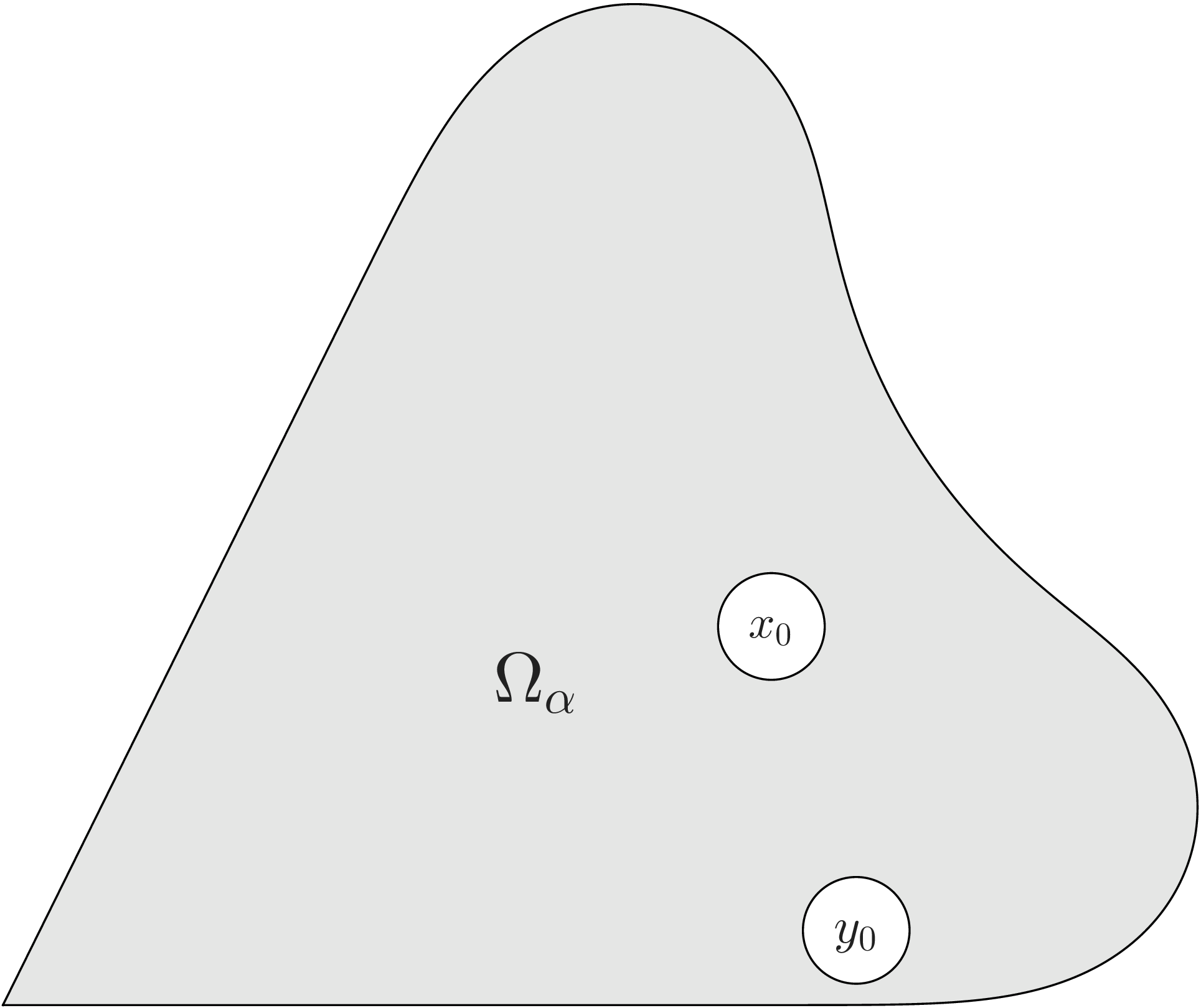}
  \end{minipage}
    \caption{The domains $\Omega$ and $\Omega_{\alpha}$ in an example where $\alpha=\frac{\pi}{2\sqrt{2}}$ with the holes used in Theorem~\ref{t:gap} removed.}
    \label{f:OmegaAlpha}
\end{figure}

On curvilinear polygons, $\mc{P}$, the article~\cite{GaMaWa:26} also proves an asymptotic formula for the Steklov spectrum of the form
\begin{equation} \label{e:dirac}
|\sigma_j-\lambda_j(\mc{D})|\leq C_\e\sigma_j^{-\frac{2\pi}{\alpha_{\max}}+\e},
\end{equation}
where $\mc{D}$ is a Dirac operator on the boundary of $\mc{P}$ that is constructed using only local information about $\partial\mc{P}$. A second goal of this article is to show that~\eqref{e:dirac} is optimal modulo an $\epsilon$ loss in the sense that there are two curvilinear polygons whose boundaries have identical local information and whose Steklov spectra differ by at least $\sigma_j^{-\frac{2\pi}{\alpha}}$.

To do this,  we consider puncturing $\Omega_{\alpha}$ by two different holes of the same radius, $H_1$ centered at $x_0$ and $H_0$ at $y_0$. We show that, despite the fact that \emph{all} local information about $\Omega_{\alpha}\setminus H_1$ is the same as that for $\Omega_{\alpha}\setminus H_0$, the high frequency spectrum is substantially perturbed by the movement of $H_1$ to $H_0$. 
\begin{theorem}\label{t:gap}
    Let $\alpha\in (0,2\pi)\setminus \mc{E}$. Then for all $r_0>0$, there is $\e_0>0$ such that for all  $x_0\in \Omega_{\alpha}$ with $d(x_0,\partial\Omega_{\alpha})>r_0$ and all $0<\e<\e_0$ there is $y_0\in \Omega_\alpha$ with $B(y_0,\e)\subset \Omega_{\alpha}$ such that 
    $$
    \limsup_{j\to \infty} \Big(\sigma_j\big(\Lambda^{\Omega_{\alpha}\setminus B(x_0,\e)}\big)+\sigma_j\big(\Lambda^{\Omega_{\alpha}\setminus B(y_0,\e)}\big)\Big)^{\frac{2\pi }{\alpha}}\Big|\sigma_j\big(\Lambda^{\Omega_{\alpha}\setminus B(x_0,\e)}\big)-\sigma_j\big(\Lambda^{\Omega_{\alpha}\setminus B(y_0,\e)}\big)\Big|>0.
    $$
\end{theorem}
\begin{remark}
    It is easy to see that for any curvilinear polygon,  $cj\leq\sigma_j\leq Cj$  (see e.g. Lemma~\ref{l:basic}) and hence that one may replace 
    $$
    \sigma_j\big(\Lambda^{\Omega_{\alpha}\setminus B(x_0,\e)}\big)+\sigma_j\big(\Lambda^{\Omega_{\alpha}\setminus B(y_0,\e)}\big)
    $$
    in Theorem~\ref{t:gap} and $\sigma_j$ in Theorem~\ref{t:lowerBound} by $j$ without changing the statement. 
\end{remark}
\begin{remark}
    It is not hard to see from the proof of Theorem~\ref{t:gap} that the shape of the hole cut in $\Omega_{\alpha}$ is unimportant. What is used in the proof is, in fact, that the hole has diameter $\sim \e$ and there is $N>0$ such that the hole  has volume larger than $\e^N$. 
\end{remark}
\begin{figure}
    \centering
    \includegraphics[width=0.5\linewidth]{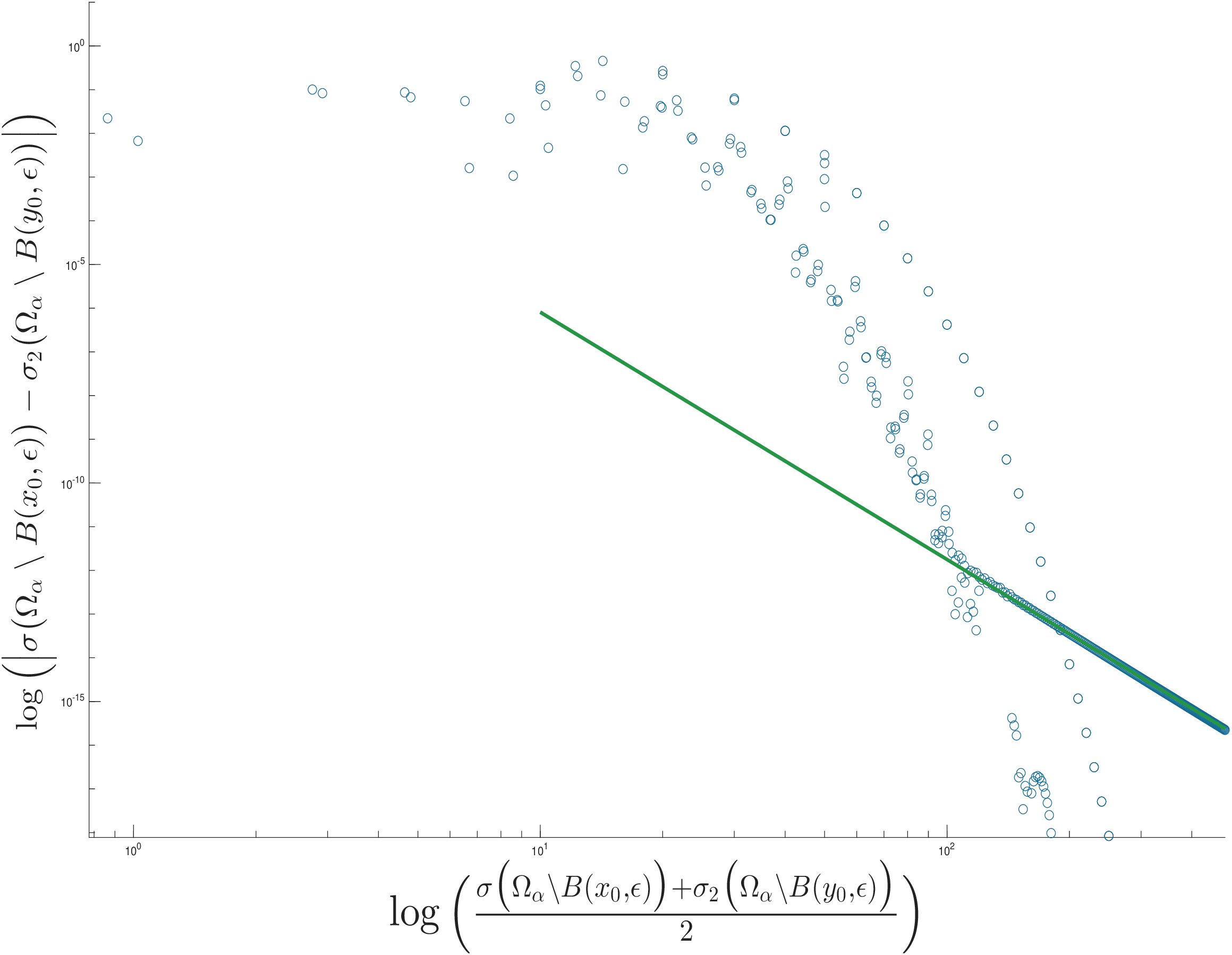}
    \caption{Log-log plot of the numerically computed differences $|\sigma_j\big(\Lambda^{\Omega_{\alpha}\setminus B(x_0,\e)}\big)-\sigma_j\big(\Lambda^{\Omega_{\alpha}\setminus B(y_0,\e)}\big)|$, with the domain $\Omega_{\alpha}$ pictured in Figure~\ref{f:OmegaAlpha}, $y_0=0.7+0.25i$ and $x_0=e^{i\pi/(3\sqrt{2})}y_0$ and $\epsilon = 0.1$. The solid green line is a line with slope $-4\sqrt{2}=-2\pi/\alpha$. This line has the same slope as the lower bound from Theorem~\ref{t:gap} and, up to an $\epsilon$ loss, the upper bound from~\cite[Theorem 1.1]{GaMaWa:26} }
    \label{fig:placeholder}
\end{figure}

{\noindent\bf{Outline of the proof} }\ 

The proofs of Theorem~\ref{t:lowerBound} and~\ref{t:gap} both rely on a trace lower bound together with a variational principle that identifies the quadratic form associated with the difference between the DtN map on a domain $\Omega$ with and without a hole, $H\Subset \Omega$. The general philosophy that lower bounds on a relative trace can be used to understand spectra has many precedents (see e.g.~\cite{davies2013inverse,smith2016heat,christiansen2017sharp, molchanov2009general}).

In this paper, we consider the difference
$$
\mathscr{D}^{\flat}(\Omega,H):= 1_{\partial\Omega}\Lambda^{\Omega\setminus H}1_{\partial\Omega}-\Lambda^{\Omega},
$$
and show in Lemma~\ref{l:differenceDescribe} that, given $f\in H^{\frac{1}{2}}(\partial\Omega)$, 
$
\langle \mathscr{D}^{\flat}(\Omega,H)f,f\rangle 
$
is the $\mc{H}^{\Omega}f$ capacity of $H$ in $\Omega$ (see e.g.~\cite{abatangelo2019spectral} for another use of this quantity in spectral estimates).

From this, we obtain two important consequences. First, we can control
\begin{equation}
\label{e:kernelDescribe}
\langle \mathscr{D}^{\flat}(\Omega,f_{\alpha}^{-1}(H))\delta_x,\delta_x\rangle 
\end{equation}
and second, if $f=\phi_{\sigma_j}$ is a Steklov eigenfunction on $\partial \Omega_{\alpha}$, then
\begin{equation}
\label{e:differenceLower}
\langle \mathscr{D}^{\flat}(\Omega_{\alpha},H)\phi_{\sigma_j},\phi_{\sigma_j}\rangle \leq C\|u_{\sigma_j}\|_{L^2(\tilde{H})}^2,
\end{equation}
where $H\Subset\tilde{H}$.

With Theorem~\ref{t:lowerBound} in mind, the inequality~\eqref{e:differenceLower} makes finding a lower bound on a certain trace of $\mathscr{D}^{\flat}(\Omega_\alpha,H)$ natural. Specifically, to obtain a lower bound on $\|u_{\sigma_j}\|_{L^2(\tilde{H})}^2$, we would ideally like to find a lower bound for
$$
\sum_j 1_{[1,\infty)}(h\sigma_j)\langle \mathscr{D}^{\flat}(\Omega_{\alpha},H)\phi_{\sigma_j},\phi_{\sigma_j}\rangle, \qquad 0<h<1.
$$
However, as is usually the case, it is difficult to control sharp spectral cutoffs and we instead use a smooth approximation to $1_{[1,\infty)}$. 

Rather than convolving $1_{[1,\infty)}$ with a smooth function, it turns out to be simpler to consider for some $m>0$ the quantity 
\begin{align*}
\sum_j \psi((h\sigma_j)^m)\langle \mathscr{D}^{\flat}(\Omega_{\alpha},H)\phi_{\sigma_j},\phi_{\sigma_j}\rangle&=\sum_j \Big\langle \psi((h\Lambda^{\Omega_{\alpha}})^m)\ \mathscr{D}^{\flat}(\Omega_{\alpha},H)\phi_{\sigma_j},\phi_{\sigma_j}\Big\rangle\\
&=\tr \Big(\psi((h\Lambda^{\Omega_{\alpha}})^m)\mathscr{D}^{\flat}(\Omega_{\alpha},H)\Big),
\end{align*}
with $\psi(x):=\frac{x}{1+x}$.
This regularization is particularly convenient for proving a lower bound on the trace because it allows us to work directly with powers of $\Lambda^{\Omega_{\alpha}}$ instead of using the full functional calculus. (This type of regularization is also known as a \emph{scaled Yosida approximation of $(h\Lambda^{\Omega_{\alpha}})^m$}~\cite[Definition before Proposition 7.2]{brezis2011functional}.)

To finish the proof, we show in Lemmas~\ref{l:traceLower} and~\ref{l:traceLower2} that for certain $m$
\begin{equation}
\label{e:traceLowerTemp}
ch^{\frac{2\pi}{\alpha}-1}\leq \tr \Big(\psi((h\Lambda^{\Omega_{\alpha}})^m)\mathscr{D}^{\flat}(\Omega_{\alpha},H)\Big).
\end{equation}
This relies on~\eqref{e:kernelDescribe} which, together with the conformal invariance of harmonic functions, allows us to show that the most singular part of the kernel of  $\mathscr{D}^{\flat}(\Omega_{\alpha},H)$ has the form $c|s|^{\frac{\pi}{\alpha}-1}|t|^{\frac{\pi}{\alpha}-1}$. By finding approximations to $\big(\Lambda^{\Omega_\alpha}\big)^m|s|^{\frac{\pi}{\alpha}-1}$ in Corollary~\ref{c:exact}, and understanding the domain of $(\Lambda^{\Omega_{\alpha}})^m$ sufficiently well, we are then able to prove~\eqref{e:traceLowerTemp} which, together with an elementary contradiction argument, proves Theorem~\ref{t:lowerBound}.

The proof of Theorem~\ref{t:gap} relies on a similar argument except that we replace $\psi((h\Lambda^{\Omega_{\alpha})^m})$ by $ \psi((h\Lambda_\theta)^m)$, with 
$$
\Lambda_{\theta}= \theta\Lambda^{\Omega_{\alpha}\setminus H_1}+(1-\theta )\Lambda^{\Omega_{\alpha}\setminus H_0},
$$
and integrate the lower bound in $\theta$ to obtain
\begin{gather*}
ch^{\frac{2\pi}{\alpha}}\leq \tr\Big(\Phi_m\big( h\Lambda^{\Omega_{\alpha}\setminus H_1}\big)-\Phi_m\big(h\Lambda^{\Omega_{\alpha}\setminus H_0}\big)\Big)=\sum_j\Phi_m\big(h\sigma_j(\Lambda^{\Omega_{\alpha}\setminus H_1})\big)-\Phi_m\big(h\sigma_j(\Lambda^{\Omega_{\alpha}\setminus H_0})\big),\\ \Phi_m(x):=\int_0^x\psi(s^m)ds.
\end{gather*}
To prove Theorem~\ref{t:gap} we then again proceed by contradiction, using that fact that $\Phi_m(x)\sim x$ for $x$ large and $|\Phi_m(x+s)-\Phi_m(x)|\leq Cx^m|s|$ for  $0<|s|<x$ small.

\medskip
\noindent {\textsc{Acknowledgements:}} JG was supported by ERC Synergy Grant ``PSINumScat"  (101167139), EPSRC grants EP/V001760/1 and EP/V051636/1, and Leverhulme Research Project Grant RPG-2023-325. The proof in this text resulted from a discussion with ChatGPT 6 Astra as described in the abstract. In addition, ChatGPT was used to generate the numerical codes which produce the figures and for proofreading.   Thanks also to Leonid Parnovski for helpful discussions.
\medskip

\section{Comparison of Dirichlet to Neumann Maps on domains with and without holes}
\label{s:comparison}
In this section, we let $\Omega\subset \mathbb{R}^d$ be open, connected with Lipschitz boundary and $H\Subset \Omega$ open with Lipschitz boundary. We then compare the DtN map on $\Omega$ with the restriction of that on $\Omega\setminus H$ restricted to $\partial\Omega$. More precisely, we study the operator
$$
\mathscr{D}^{\flat}(\Omega,H):=1_{\partial\Omega}\Lambda^{\Omega\setminus H}1_{\partial\Omega}-\Lambda^{\Omega},
$$
where $1_{\partial\Omega}$ indicates restriction to $\partial\Omega$.

\subsection{Variational formulae for $\mathscr{D}^{\flat}(\Omega,H)$}
\begin{lemma}
\label{l:differenceDescribe}
For all $f\in H^{1/2}(\partial\Omega)$, 
$$
\big\langle \mathscr{D}^{\flat}(\Omega,H)f,f\big\rangle=\min\Big\{\int_\Omega |\nabla v|^2dx\,:\, v\in H_0^1(\Omega),\, v=-\mathcal{H}^{\Omega}f\text{ on }H\Big\},
$$
\end{lemma}
\begin{proof}
Let $f\in H^{1/2}(\partial\Omega)$ and set
    $$
    w_f:=1_{\Omega\setminus H}\mathcal{H}^{\Omega_H}1_{\partial\Omega}f- \mathcal{H}^\Omega f.
    $$
    Then, $w$ is harmonic in $\Omega\setminus H$ and hence for all $g\in H_0^1(\Omega\setminus H)$,
    \begin{equation}
    \label{e:harmonic}
    \int_{\Omega\setminus H} \nabla w_f\cdot \nabla gdx=0.
    \end{equation}
    Now, notice that 
    \begin{align*}
    \langle \mathscr{D}^{\flat}(\Omega,H)f,f\rangle& = \int_{\Omega_H}\nabla \mathcal{H}^{\Omega_H}1_{\Omega}f\cdot \nabla \mathcal{H}^{\Omega_H}1_{\Omega}fdx-\int_{\Omega}\nabla \mathcal{H}^{\Omega}f\cdot \nabla \mathcal{H}^{\Omega}fdx=\int_{\Omega}|\nabla w_f|^2dx\\
    &=\min\Big\{\int_\Omega |\nabla v|^2dx\,:\, v\in H_0^1(\Omega),\, v=-\mathcal{H}^{\Omega}f\text{ on }H\Big\},
    \end{align*}
    where the last equality follows from~\eqref{e:harmonic}. 
\end{proof}

 Our next lemma specializes to the case of $\Omega $ and $H$ having smooth boundary.
\begin{lemma}
\label{l:basicKernel}
    Suppose that $\Omega$ and $H$ have smooth boundary. Then $\mathscr{D}^{\flat}(\Omega,H)$ is smoothing and its kernel, $K_{\mathscr{D}^{\flat}(\Omega,H)}$, satisfies
    \begin{equation} 
    \label{e:minimization}
    K_{\mathscr{D}^{\flat}(\Omega,H)}(x,x)=\min\Big\{ \int_{\Omega}|\nabla v|^2dx\,:\, v\in H_0^1(\Omega),\, v|_{H}=-P_{\Omega}(\cdot,x)|_{H}\Big\},
    \end{equation}
    where $P_{\Omega}$ denotes the kernel of $\mc{H}^\Omega$.
\end{lemma}
\begin{proof}
The smoothing property of $\mathscr{D}^{\flat}(\Omega,H)$ follows from the fact that
\begin{gather*}
-\Delta (\mathcal{H}^{\Omega\setminus H}f-\mathcal{H}^{\Omega}f)=0\text{ in }\Omega\setminus H,\\
(\mathcal{H}^{\Omega\setminus H}-\mathcal{H}^{\Omega})f|_{\partial\Omega}=0,\qquad (\mathcal{H}^{\Omega\setminus H}-\mathcal{H}^{\Omega})f|_{\partial H}= -\mathcal{H}^{\Omega}f|_{\partial H}
\end{gather*}
and, by elliptic regularity $\mathcal{H}^{\Omega}f|_{\partial H}\in C^\infty(\partial H)$ and hence, again using elliptic regularity $\mathcal{H}^{\Omega\setminus H}f-\mathcal{H}^{\Omega}f)\in H^N(\Omega\setminus H)$ for any $N$.

    The lemma now follows from letting $f_\e\in H^{1/2}(\partial\Omega)$ with $f_\e \to \delta_x$ (using the fact that, since $H\Subset \Omega$, $\mathcal{H}^{\Omega}f_\e|_{H}\to P_{\Omega}(\cdot, x)|_{H}$ in $H^{1}$). 
\end{proof}

\subsection{Holes with large and small difference kernels}

We first show that if the hole $H$ is close to the boundary, then $K_{\mathscr{D}^{\flat}(\Omega,H)}(x_0,x_0)$ is small.
\begin{lemma}
\label{l:closeToBoundary}
    Let $\Omega$ have smooth boundary and $x_0\in \partial\Omega$ and $c_0>0$. Then there is $C>0$ such that for all $0<\e$ and $H$ with smooth boundary satisfying 
    $$
    H\Subset \{ x\in \Omega\,:\, d(x,x_0)>c_0,\, d(x,\partial\Omega)<2\e\},
    $$
    $$
    0\leq K_{\mathscr{D}^{\flat}(\Omega,H)}(x_0,x_0)\leq C\e.
    $$
\end{lemma}
\begin{proof}
    Let $\chi \in C^\infty(\Omega;[0,1])$ with $\supp (1-\chi)\cap H=\emptyset$, $d(\supp \chi, x_0)>c_0/2$ and 
    $$
    \supp \chi\subset \mathcal{C}_\e,\qquad \|\nabla \chi \|_{L^\infty}\leq C\e^{-1}
    $$
        where 
    $$
    \mathcal{C}_\e:=\{ y\in \Omega\,:\, d(y,x_0)\geq c_0/2,\, d(y,\partial\Omega)\leq 3\e\}.
    $$
    (It is easy to construct such a function using Fermi normal coordinates near $\partial\Omega$). Then, observe that $-\chi P_{\Omega}(\cdot, x_0)$ is a candidate for~\eqref{e:minimization} and
    $$
    \int |\nabla \chi P_{\Omega}(y,x_0)|^2dy\leq C\e^{-2}\int_{\mc{C}_\e}|P_{\Omega}(y,x_0)|^2dy+ C\int \chi^2|\nabla P_{\Omega}(y,x_0)|^2dy\leq C\e.
    $$
    Here, we use that since $d(\supp \chi,x_0)\geq c_0/2$, $P_{\Omega}(y,x_0)$ is smooth in $y$ and $P(\cdot ,x_0)|_{\partial\Omega\setminus\{x_0\}}=0$ which implies that,  $\|P_{\Omega}(\cdot, x_0)\|_{L^\infty(\mathcal{C}_\e)}\leq C\e$. 
\end{proof}

Next, we show that if the hole $H$ is not close to the boundary, then $K_{\mathscr{D}^{\flat}(\Omega,H)}(x_0,x_0)$ is controlled from below.
\begin{lemma}
\label{l:kernelLower}
    Let $\Omega\subset \mathbb{R}^2$ have smooth boundary, $x_0\in\partial\Omega$, $c_0>0$, $1\leq p<\infty$. Then there is $c>0$ such that for all $H$ with smooth boundary satisfying
    $$
    H\subset \{ d(x,\partial\Omega)>c_0\},
    $$
     we have
    $$
    K_{\mathscr{D}^{\flat}(\Omega,H)}(x_0,x_0)\geq c(\vol(H))^{2/p}.
    $$
    \end{lemma}
    \begin{proof}
    Suppose that $v$ is a candidate for~\eqref{e:minimization}. Then, 
    $$
    v|_{H}=-P_{\Omega}(\cdot, x_0)|_{H}
    $$
    and, hence, since $P_{\Omega}(\cdot, x_0)>0$ in $\Omega$ there is $c>0$ satisfying
    $$
    c(\vol(H))^{1/p}<\inf_{x\in H}P_{\Omega}(x, x_0)\vol(H)^{1/p}\leq \|v\|_{L^p(H)}\leq\| v\|_{L^p(\Omega)}\leq C\|v\|_{H^1(\Omega)}\leq C\|\nabla v\|_{L^2(\Omega)}.
    $$
    Hence, the lemma follows from~\eqref{e:minimization}.
    \end{proof}

     \section{Kernels of DtN differences in $\Omega_{\alpha}$}

    From now on, we identify $\partial\Omega_{\alpha}$ with $\mathbb{S}^1$ using a constant speed, bijective map $\gamma_\alpha:\mathbb{S}^1\to \partial\Omega_{\alpha}$ such that $\gamma_{\alpha}(0)=0$. 
Then, let 
   $$
   w_{\alpha}:\mathbb{S}^1\to \to \mathbb{R},\qquad x_{\alpha}:\mathbb{S}^1\to \mathbb{S}^1
   $$
   such that $x_{\alpha},w_{\alpha}$ are smooth away from $0$, $w_{\alpha}>0$ away from $0$, $x_{\alpha}$ is a homeomorphism, smooth away from $0$, and, in a neighborhood of $0$ are given by
   $$
   x_{\alpha}(s)=\sgn(s)|s|^{\frac{\pi}{\alpha}},\qquad w_{\alpha}(s)=|s|^{\frac{\pi}{\alpha}-1}. 
   $$

   \begin{lemma}\label{l:kernelForm}
       Let $H\Subset \Omega_{\alpha}$ have smooth boundary. Then, the kernel of $\mathscr{D}^{\flat}(\Omega_{\alpha},H)$ has the form
       \begin{equation}
       \label{e:kernelForm1}
       k_{\mathscr{D}^{\flat}(\Omega_{\alpha},H)}(s,t)= w_\alpha(t)k(x_\alpha(t),x_\alpha(s))w_{\alpha}(s),
       \end{equation}
       where $k\in C^\infty(\mathbb{S}^1\times \mathbb{S}^1)$ and $k(0,0)= ck_{\mathscr{D}^{\flat}(\Omega, f_{\alpha}^{-1}(H))}(0,0)$ and $c>0$
   \end{lemma}
   \begin{proof}
       Observe that $\mathscr{D}^{\flat}(\Omega,f_{\alpha}^{-1}(H))$ is smoothing. Therefore, since 
       $$
       \mathscr{D}^{\flat}(\Omega_\alpha,H)= \frac{\pi}{\alpha}|z|^{\frac{\pi}{\alpha}-1} f_{\alpha}^*\mathscr{D}^{\flat}(\Omega, f_{\alpha}^{-1}(H))(f_{\alpha}^*)^{-1},
       $$
       the lemma follows.
   \end{proof}

   Next let $H_1,H_0\Subset \Omega_{\alpha}$ be simply connected and equal length $L$. We again identify the boundaries of $H_i$ with $\mathbb{S}^1$ using maps $\gamma_i:\mathbb{S}^1\to \partial H_i$ so that $\partial(\Omega_{\alpha}\setminus H_i)\cong \mathbb{S}^1\sqcup \mathbb{S}^1$ and hence $L^2(\partial(\Omega_{\alpha}\setminus H_i))\cong L^2(\mathbb{S}^1;\mathbb{C}^2)$, where we identify the first entry with $\partial H_i$ and the second with $\partial\Omega_{\alpha}$. 
      
   An argument almost identical to that for Lemma~\ref{l:kernelForm} yields the following
   \begin{lemma}
   \label{l:kernelForm2}
       Let $H_i\Subset \Omega_{\alpha}$, $i=1,2$ be open, simply connected, possibly empty, with smooth boundaries both having length $L\geq 0$. Then the kernel of $\mathscr{D}(H_1,H_0):=\Lambda^{\Omega_\alpha\setminus H_1}-\Lambda^{\Omega_{\alpha}\setminus H_0}$ has the form
       \begin{equation}
       \label{e:kernelForm}
        K_{\mathscr{D}(H_1,H_0)}=\begin{pmatrix} K_{HH}(s_H,t_H)&K_{H\Omega }(s_H,x_\alpha(t))w_\alpha(t)\\\overline{K}_{H \Omega}(t_H,x_\alpha(s))w_\alpha(s)&K_{\Omega\Omega}(x_\alpha(t),x_\alpha(s))w_{\alpha}(s)w_{\alpha}(t)\end{pmatrix},
       \end{equation}
       where $K_{HH}$, $K_{H\Omega}$, $K_{\Omega\Omega}\in C^\infty(\mathbb{S}^1\times \mathbb{S}^1)$ and $K_{\Omega\Omega}(0,0)=\frac{\pi^2}{\alpha^2}(k_{\mathscr{D}^{\flat}(\Omega,f_{\alpha}^{-1}(H_1))}(0,0)-k_{\mathscr{D}^{\flat}(\Omega,f_{\alpha}^{-1}(H_0))}(0,0))$.
   \end{lemma}
    
    \section{Trace class estimates on DtN differences in the plane }

In this section, we study regularised traces of $\mathscr{D}^\flat(\Omega, H)$ and $\mathscr{D}(H_1,H_0)$. 

\subsection{Trace class properties of operators with kernels of the form~\eqref{e:kernelForm}}
We first show that these operators are, in fact, trace class. To do this, we recall the following basic fact:
\begin{lemma}
\label{l:regularity}
    Let $\beta>-\frac{1}{2}$, $\alpha>0$ and $f\in C_c^\infty(-1,1)$, then
    $$
    |t|^\beta f(\sgn (t)|t|^\alpha)\in H^{r},\qquad 
    r<\beta+\frac{1}{2},\qquad f(\sgn(t)|t|^{\alpha})\in H^
    r,\qquad r<\alpha+\frac{1}{2}
    $$
\end{lemma}

                \begin{lemma}
                \label{l:traceClass}
            Suppose that $k:\mathbb{S}^1\times \mathbb{S}^1\to \mathbb{C}$ is smooth. Then the operators $B_i$ $i=1,2,3$ given by
            \begin{align*}
            B_1u(t)&= \int k(x_{\alpha}(t),x_{\alpha}(s))w_\alpha(s)w_\alpha(t)u(s)ds\\
            B_2u(t)&= \int k(x_{\alpha}(t),x_{\alpha}(s))w_\alpha(s)u(s)ds\\
            B_3u(t)&= \int k(x_{\alpha}(t),x_{\alpha}(s))w_\alpha(t)u(s)ds
            \end{align*}
            are trace class.
        \end{lemma}
        \begin{proof}
        We consider $B_1$, the other cases being similar.
            Taylor expanding $k$ at $(0,\tilde{s})$, we obtain
            $$
            k(t,s)= \sum_{j=0}^{M-1}t^ja_j(s)+ t^Me_M(t,s),
            $$
            where $a_j$ and $e_M$ are smooth. 
            Therefore, 
            $$
            B_1= \sum_{j=0}^{M-1}x_\alpha^jw_\alpha \otimes a_j(x_\Omega)w_\alpha + E_M,
            $$
            where 
            $$
            \|E_M\|_{L^2\to H^{(M+1)\frac{\pi}{\alpha}-\frac{1}{2}-0}}<\infty.
            $$
            Taking $M$ large enough that $(M+1)\frac{\pi}{\alpha}-\frac{1}{2}>1$ implies that $E_M$ is trace class and hence, since $E_M-B_1$ is finite rank, so is $B_1$. 
        \end{proof}

\subsection{The action of DtN maps on space of functions with $b$-regularity}

We first recall a direct consequence of~\cite[Proposition 3.9, Theorem 5.4]{BaNiZi:05} that allows us to understand the action of DtN maps on functions with $b$-regularity.
\begin{lemma}
\label{l:rough}
    Let $H\Subset \Omega_{\alpha}$ be open, possibly empty, with smooth, connected boundary, $\frac{1}{2}-\frac{\pi}{\alpha}<t<\frac{\pi}{\alpha}+\frac{1}{2}$. Then  
    \begin{equation} \label{e:rough}
    \| |s|^{m-t}\partial_s^mg\|_{L^2}<\infty\text{ for all }m\geq 0\quad \Rightarrow \quad 
    \||s|^{m-t+1}\partial_s^m\Lambda^{\Omega_{\alpha}\setminus H}1_{\partial\Omega_{\alpha}}g\|_{L^2}<\infty,\text{ for all }m\geq 0.
    \end{equation}
\end{lemma}

A simple application of Lemma~\ref{l:rough} together with the characterisation of the domains of small powers of $\Lambda^{\Omega_{\alpha}\setminus H}$ in Corollary~\ref{c:domains} allows us to apply the following Lemma to DtN maps and hence to show that functions with $b$-regularity are in the domain of high powers of DtN maps.
\begin{lemma}
\label{l:highDomain}
Suppose that $\Lambda\geq 0$ is self-adjoint with $\Lambda:H^1\to L^2$ and satisfies~\eqref{e:rough} for $0\leq t<\frac{\pi}{\alpha}+\frac{1}{2}$. Then for all $0\leq t<\frac{\pi}{\alpha}+\frac{1}{2}$, $g$ satisfying
    $$
    \||s|^{m-t}\partial_s^mg\|_{L^2}<\infty,\qquad m\geq 0,
    $$
    we have $\Lambda^tg\in L^2$.
\end{lemma}
\begin{proof}
    Let $J:=\lfloor t\rfloor$. Then, by~\eqref{e:rough}
    $$
    \||s|^{m-t+J}\partial_s^m\Lambda^J1_{\partial\Omega_{\alpha}}g\|_{L^2}<\infty.
    $$
    Therefore, it is enough to show that if $0<t<1$ and for all $m\geq 0$,
    $$
    \||s|^{m-t}\partial_s^m g\|_{L^2}<\infty
    $$
    then, $\Lambda^tg\in L^2$. 
    Since $\Lambda:H^1\to L^2$, we have by interpolation that $\big(\Lambda^{\Omega_{\alpha}\setminus H}\big)^t:H^t\to L^2$.

We now show that $g\in H^t$ and hence $(\Lambda^{\Omega_{\alpha}\setminus H})^tg\in L^2$ as claimed.  

 Since $\psi g\in C^\infty$ for any $\psi\in C^\infty$ supported away from $0$, we may assume that $g$ is supported in a small neighborhood of $0$. Therefore, 
$$
    \|g\|_{H^t}^2=\|g\|_{L^2}+\int_{-\delta}^\delta\int_{-\delta}^\delta \frac{|g(x)-g(y)|^2}{|x-y|^{2t+1}}dxdy
    $$
We now estimate
    \begin{align*}
       &\iint \frac{|g(x)-g(y)|^2}{|x-y|^{2t+1}}dxdy\\
       &= \iint_{\substack{|x-y|\leq \frac{|y|}{2}}}\frac{|g(x)-g(y)|^2}{|x-y|^{2t+1}}dxdy+ \iint_{ \frac{|y|}{2}< |x-y|}\frac{|g(x)-g(y)|^2}{|x-y|^{2t+1}}dxdy\\
              &\leq  \iint_{|x-y|\leq \frac{|y|}{2}}\frac{\int_{0}^1 |g'(y+s(x-y))|^2ds}{|x-y|^{2t-1}}dxdy+ 2\iint_{ \frac{|y|}{2}< |x-y|}\frac{|g(x)|^2+|g(y)|^2}{|x-y|^{2t+1}}dxdy\\
    &\leq \int_{0}^{1}\iint_{|x-y|\leq \frac{|y|}{2}} \frac{ |g'(
    y+s(x-y)
    )|^2}{|x-y|^{2t-1}}dxdyds 
  +C\int |g(x)|^2|x|^{-2t}dx.
     %&\leq C\int_{0}^{1/2}\iint_{|(1+s)x-w)|\leq \frac{(1-s)|x|}{2}}\frac{ |g'(w)
    % |^2}{|(1+s)x-w|^{2t-1}}dxdwds +C\int_{1/2}^1\iint_{|w-sy|\leq \frac{s|y|}{2}}\frac{\int |g'(
    % w
    % )|^2}{|w-sy|^{2t-1}}dwdyds \\
    % &\qquad +C\int |g(x)|^2|x|^{-2t}dx\\
    % &\leq C\int_{0}^{1/2}\int \int_{c|x|}^{C|x|}\iint_{|(1+s)x-w)|\leq \frac{(1-s)|x|}{2}}\frac{ |g'(w)
    % |^2}{|(1+s)x-w|^{2t-1}}dwdxds +C\int_{1/2}^1\iint_{|w-sy|\leq \frac{s|y|}{2}}\frac{\int |g'(
    % w
    % )|^2}{|w-sy|^{2t-1}}dwdyds \\
    % &\qquad +C\int |g(x)|^2|x|^{-2t}dx,
    \end{align*}
    where the last line follows from the fact that $|x-y|\geq \frac{|y|}{2}$ implies $|x-y|\geq c|x|$.

    Now, for $0\leq s\leq 1$
    \begin{align*}
        \iint_{|x-y|\leq \frac{|y|}{2}} \frac{ |g'(
    y+s(x-y)
    )|^2}{|x-y|^{2t-1}}dxdy&= \iint_{|v|\leq \frac{|w-sv|}{2}} \frac{ |g'(
    w
    )|^2}{|v|^{2t-1}}dvdw\\
    &= \int \sgn(w)\int_{-\frac{w}{2-s}}^{\frac{w}{2+s}} \frac{ |g'(
    w
    )|^2}{|v|^{2t-1}}dvdw\\
     &\leq C \int  |g'(w)|^2 |w|^{2-2t}dw,
    \end{align*}
    and hence $g\in H^t$ as claimed.
\end{proof}

In our lower bound on the regularized trace of difference of DtN maps, we need a more precise formula for the action of DtN maps on one particular function.
\begin{lemma}
\label{l:exact}
Let $H\Subset \Omega_{\alpha}$ be open, possibly empty, with smooth, connected boundary.
    Then there is $\e>0$ such that for $\psi\in C_c^\infty(-\e,\e)$ with $\psi\equiv 1$ near $0$ and $q>\frac{1}{2}$, $q\alpha \notin \pi (2\mathbb{Z}+1)$. Then, for any $|t|<\frac{\pi}{\alpha}-\frac{1}{2}$ and $m\geq 0$,
    \begin{gather}
    \label{e:oneCorner}
    \||s|^{m-t}\partial_s^m(\Lambda^{\Omega_\alpha\setminus H} 1_{\partial\Omega_{\alpha}}\psi(|s|)|s|^{q} +q\psi(|s|)|s|^{q-1}\tan(q\alpha/2))\|_{L^2}<\infty,
    %\\
    %\||s|^{m-t}\partial_s^m(\Lambda^{\Omega_\alpha\setminus H} 1_{\partial\Omega_{\alpha}}\sgn(s)\psi(|s|)|s|^{q} -q\psi(|s|)\sgn(s)|s|^{q-1}\cot(q\alpha/2))\|_{L^2}<\infty.\label{e:twoCorner}
    \end{gather}
\end{lemma}
\begin{proof}
It will be convenient to rotate coordinates so that, with $\gamma:[-L/2,L/2)\to \partial\Omega_{\alpha}$, a unit speed parametrisation satisfying $\gamma(0)=0$, we have $\gamma'(0+)=e^{-i\alpha/2}$ and $\gamma'(0-)=-e^{i\alpha/2}$. 
Note that, by assumption there is $\e>0$ such that $\gamma(\pm s)=\pm s\gamma'(0\pm )$ on $0\leq s\leq \e$. 

Let $g:=\psi(|s|)|s|^{q}\in L^2(\partial\Omega_\alpha)$, and define
    $$
    u_q:=  \frac{r^{q}}{2\cos (q\alpha/2)}( e^{iq\theta}+e^{-iq\theta}).
    $$
    Then, 
    \begin{gather*}
    -\Delta (\mc{H}^{\Omega_\alpha\setminus H}1_{\partial\Omega_{\alpha}}g-\psi(r)u_q)= [\Delta,\psi(r)]u_q\text{ in }\Omega_{\alpha},\qquad (g-\psi u_q)|_{\partial(\Omega_{\alpha}\setminus H)}= 0
     \end{gather*}

    Therefore, by~\cite[Theorems 3.7, 5.4]{BaNiZi:05}, for $|t|<\frac{\pi}{\alpha}$ and any $m\geq 1$,
    $$
    \sum_{|\alpha|\leq m}\|r^{|\alpha|-t-1}\partial^{|\alpha|}(\mc{H}^{\Omega_\alpha\setminus H}1_{\partial\Omega_{\alpha}}g-\psi(r)u_q)\|_{L^2}\leq C_m
    $$
    Hence, for $|t|<\frac{\pi}{\alpha}-\frac{1}{2}$ and $m\geq 0$, we have
    $$
    \| |s|^{m-t}\partial_s^m(\Lambda^{\Omega_\alpha\setminus H}1_{\partial\Omega_{\alpha}}g+q\psi(|s|)|s|^{q-1}\tan (q\alpha/2))\|_{L^2}<\infty
    $$
    as claimed.
\end{proof}

As a corollary of Lemmas~\ref{l:exact} and ~\ref{l:rough}, and their analogs on $\Omega_{\alpha}\setminus H$, we obtain 
\begin{corollary}
\label{c:exact}
    Let $\psi\in C_c^\infty(-1,1)$ with $\psi\equiv 1$ near $0$, $\frac{1}{2}<q<\frac{\pi}{\alpha}$, and $m\in\mathbb{Z}$ with $0\leq m<q+\frac{1}{2}$.  Then, for all $H\Subset \Omega_{\alpha}$ with smooth boundary (possibly empty), and all $j$ and $t<\frac{\pi}{\alpha}+\frac{1}{2}$
    \begin{multline*}
    \Bigg\||s|^{j-t+m}\partial_s^j\bigg((\Lambda^{\Omega_\alpha\setminus H})^m1_{\partial\Omega_{\alpha}\setminus H}1_{\partial\Omega_{\alpha}}\psi(|s|)|s|^q\\-(-1)^m\psi(|s|)|s|^{q-m} \prod_{j=0}^{m-1} (q-j)\tan \big((q-j)\alpha/2\big) \bigg)\Bigg\|_{L^2}<\infty. 
    \end{multline*}
\end{corollary}

\begin{remark}
    It is possible to extend the range of $q$'s to which Lemma~\ref{l:exact} and Corollary~\ref{c:exact} apply, but to do so requires  one to solve the Dirichlet problem in spaces worse than $H^{\frac{1}{2}}$ and the use of the distributional normal derivative. To avoid these technicalities we restrict ourselves to the case $q>\frac{1}{2}$ used in this paper.
\end{remark}

\subsection{A lower bound on the regularized trace}

Next, we prove the key technical estimate used to prove Theorems~\ref{t:lowerBound} and~\ref{t:gap}. It shows that certain regularized traces of kernels of the form~\eqref{e:kernelForm} can be controlled from below. To state it, we define for $m\geq 0$,
        $$
         \psi_{m}(x):=\frac{x^m}{1+x^m}.
        $$
       \begin{lemma}
       \label{l:traceLower}
    Let $\alpha\in (0,2\pi)\setminus \mc{E}$, $L> 0$, and $m\in \mathbb{R}$ with $\frac{2\pi}{\alpha}-1<m<\frac{2\pi}{\alpha}$. Let $H_i\Subset \Omega_{\alpha}$, $i=1,2$  be open with smooth, connected boundary of length $L$. 
Then $\mathscr{D}(H_1,H_0)$ is trace class and there is $c>0$ such that if $\mathscr{D}(H_1,H_0)$ has kernel of the form~\eqref{e:kernelForm} with $K_{\Omega\Omega}(0,0)>0$, then
        \begin{gather}
        \inf_{\theta\in[0,1]}\tr\big(T^{\theta}_{h,m}\mathscr{D}(H_1,H_0))\geq c K_{\Omega\Omega}(0,0)h^{\frac{2\pi}{\alpha}-1}+o(h^{\frac{2\pi}{\alpha}-1}),\label{e:abstractTrace}\\
        T_{h,m}^\theta:=\psi_{m}(h\Lambda_\theta),\qquad \Lambda_{\theta}:=I+\theta\Lambda^{\Omega_{\alpha}\setminus H_1}+(1-\theta)\Lambda^{\Omega_{\alpha}\setminus H_0}.\nonumber
                \end{gather}
    \end{lemma}
\begin{proof}

The fact that $\mathscr{D}(H_1,H_0)$ is trace class follows from Lemma~\ref{l:traceClass}. To reduce the notation below, we write $\Lambda_i:=1+\Lambda^{\Omega_{\alpha}\setminus H_i}$, $i=1,2$.

\noindent Step 1: We claim that for $0\leq t<\min(\frac{\pi}{\alpha}+\frac{1}{2},\frac{3}{2})$,  
\begin{equation} \label{e:goodDomains}
\sup_{\theta\in[0,1]}\|\Lambda_\theta^t\|_{H^{t}\to L^2}+\|\Lambda_\theta^{-t}\|_{ L^2\to H^{t}}<\infty.
\end{equation}

First notice that, since by Corollary~\ref{c:domains} $\Lambda_1^{2t}:H^{t}\to H^{-t}$ and $\Lambda_1\geq 0$, we have
$$
\|\Lambda_1^{t}u\|_{L^2}^2=\langle \Lambda_1^{2t} u,u\rangle_{L^2}\leq C\|u\|_{H^{t}}^2.
$$
Hence, $\Lambda_1^t:H^{t}\to L^2$
and, since $\Lambda_1^{2t}:H^{t}\to H^{-t}$ is invertible,
$$
\|u\|_{H^{t}}\leq C\|\Lambda_1^{2t}u\|_{H^{-t}}\leq C\|\Lambda_1^t\|_{L^2\to H^{-t}}\|\Lambda_1^tu\|_{L^2}\leq C\|\Lambda_1^tu\|_{L^2}.
$$
In particular, $\Lambda_1^t:H^{t}\to L^2$ is invertible.

After applying the same argument to $\Lambda_\theta$, it remains to show that 
\begin{equation} \label{e:boundTheta}
\sup_{\theta\in[0,1]}\|\Lambda_\theta^{2t}\|_{H^{t}\to H^{-t}}+\|\Lambda_\theta^{-2t}\|_{H^{-t}\to H^{t}}<\infty.
\end{equation}
For this, observe that by Lemma~\ref{l:bounded}, for $-t+1\leq s\leq t$, 
$$
\|\Lambda_\theta\|_{H^s\to H^{s-1}}\leq C(\|\Lambda_1\|_{H^s\to H^{s-1}}+\|\Lambda_2\|_{H^s\to H^{s-1}})\leq C. 
$$

First, notice that since $\Lambda_i\geq 0$, we have for all $\theta\in[0,1]$, $\Lambda_{\theta}^{-1}:H^{-\frac{1}{2}}\to H^{\frac{1}{2}}$. 
Hence,
$$
\Lambda_\theta^{-1}= \Lambda_1^{-1}+\Lambda_1^{-1}(1-\theta)(\Lambda_1-\Lambda_2)\Lambda_\theta^{-1}.
$$

Therefore, for $\frac{1}{2}\leq s\leq t$, since $(\Lambda_2-\Lambda_1):H^{\frac{1}{2}}\to H^{s-1}$, 
$$
\|\Lambda_\theta^{-1}\|_{H^{s-1}\to H^s}\leq C.
$$
In particular, dualtiy, interpolation and composition imply~\eqref{e:boundTheta}.
~

\noindent Step 2: An intermediate quadratic form. 

Let $0\leq t<\min(\frac{2\pi}{\alpha}+1,3)$
$$
Q_{h,t}(u):=\langle h^t\Lambda_\theta^t(I+h^t\Lambda_\theta^t)^{-1}u,u\rangle. 
$$
Then, since 
$$
\langle (1+ch^t(1+|D|)^t)u,u\rangle \leq \|u\|_{L^2}^2+ch^t\|u\|_{H^{t/2}}^2\leq \langle (1+h^t\Lambda_\theta^t)u,u\rangle, 
$$
we have $0\leq (1+h^t\Lambda_\theta^t)^{-1}\leq (1+ch^t|D|^t)^{-1}$ and hence
\begin{equation} \label{e:quadraticLower}
\begin{aligned}
Q_{h,t}(v)&=\|v\|_{L^2}^2-\langle (1+h^t\Lambda_\theta^t)^{-1}v,v\rangle \\
&\geq \|v\|_{L^2}^2-\langle (1+ch^t(1+|D|)^t)^{-1}v,v\rangle \\
&\geq c\langle h^t(1+|D|)^t(1+ch^t(1+|D|)^t)^{-1}v,v\rangle.
\end{aligned}
\end{equation}
Similarly, since for 
$$
0\leq \langle (1+h^t\Lambda_\theta^t)u,u\rangle \leq \|u\|_{L^2}^2+Ch^t\|u\|_{H^{t/2}}^2\leq \langle (1+Ch^t(1+|D|)^t)u,u\rangle, 
$$
we have
\begin{equation} \label{e:quadraticUpper}
\begin{aligned}
Q_{h,t}(v)\leq  C\langle h^t(1+|D|)^t(1+Ch^t(1+|D|)^t)^{-1}v,v\rangle.
\end{aligned}
\end{equation}

We also claim that
\begin{equation}\label{e:quadraticCut}
Q_{h,t}(1_{h\e \Lambda_\theta\geq1 }v)\leq 2 Q_{h\e ,t}(v).
\end{equation}
Indeed, this follows from the fact the functional calculus and the fact that 
$$
\frac{x^t}{1+x^t}1_{ x\geq \e^{-1}}\leq 2\frac{\e ^t x^t}{1+\e^tx^t}.
$$

         \noindent Step 3: A lower bound for the weighted trace
         
         The goal of this step is to prove~\eqref{e:abstractTrace}, which follows from the fact that
        \begin{equation} \label{e:finalTrace}
        \operatorname{tr}(T^\theta_{h,m}\mathscr{D}(H_1,H_0))\geq  cK_{\Omega\Omega}(0,0) h^{\frac{2\pi}{\alpha}-1}+o(h^{\frac{2\pi}{\alpha}-1}).
        \end{equation}

        Let $B_0$ be the operator with kernel
        $$
        K_{B_0}=\begin{pmatrix} 0&0\\0&K_{\Omega\Omega}(0,0)w_\alpha(s)w_{\alpha}(t)\end{pmatrix}.
        $$
        We claim that
        \begin{equation} \label{e:traceEstimates}
        \tr(T^\theta_{h,m}B_0)\geq  cK_{\Omega\Omega}(0,0)h^{\frac{2\pi}{\alpha}-1}+o(h^{\frac{2\pi}{\alpha}-1}),\qquad |\tr(T_{h,m}^\theta(\mathscr{D}(H_1,H_0)-B_0)|=o (h^{\frac{2\pi}{\alpha}-1})
        \end{equation}
        
Step 3a: the principal term

We start by proving the first part of~\eqref{e:traceEstimates}. Observe that 
        $$
        \tr (T^\theta_{h,m}B_0)=K_{\Omega\Omega}(0,0) \langle T^\theta_{h,m} w_{\alpha},w_{\alpha}\rangle 
        $$

        Now, since $w_{\alpha}-c\psi(s)|s|^{\frac{\pi}{\alpha}-1}=:r\in C_c^\infty(\mathbb{S}^1\setminus \{0\})$, 
        $$
        \langle T^\theta_{h,m} w_{\alpha},w_{\alpha}\rangle \geq cQ_{h,m}\big( \psi(|s|)|s|^{\frac{\pi}{\alpha}-1}\big)  - h^m\|\Lambda_\theta^{m/2}r\|_{L^2}^2. 
        $$
        Now, by Lemmas~\ref{l:highDomain} and~\ref{l:rough}, $\Lambda_\theta^{m/2}r\in L^2$.  
        Therefore, we need only show that 
        $$
        Q_{h,m}\big( \psi(s)|s|^{\frac{\pi}{\alpha}-1}\big) \geq ch^{\frac{2\pi}{\alpha}-1}. 
        $$

        For this, let $J:=\lceil \frac{\pi}{\alpha}-\frac{3}{2}\rceil$. Then, by Lemma~\ref{l:exact} and~\eqref{e:quadraticCut},
\begin{align*}
Q_{h,m}\big(\psi(|s|)|s|^{\frac{\pi}{\alpha}-1}\big)&= \langle (1+h^m\Lambda_\theta^m)^{-1}h^{m}\Lambda_\theta^{m-2J} \Lambda_\theta^{J}\psi(|s|)|s|^{\frac{\pi}{\alpha}-1},\Lambda_\theta^{J}\psi(|s|)|s|^{\frac{\pi}{\alpha}-1}\rangle\\
&\geq c_\e h^{2J} Q_{h,m-2J}(1_{\e h \Lambda_\theta\leq 1}\Lambda_\theta^{J}\psi(|s|)|s|^{\frac{\pi}{\alpha}-1})\\
&\geq c_\e h^{2J}\big(Q_{h,m-2J}(\Lambda_\theta^{J}\psi(|s|)|s|^{\frac{\pi}{\alpha}-1})-2Q_{h\e,m-2J}(\Lambda_\theta^{J}\psi(|s|)|s|^{\frac{\pi}{\alpha}-1})\big)\\
&= c_{\e,J} h^{2J}\big(Q_{h,m-2J}(C_J\psi(|s|)|s|^{\frac{\pi}{\alpha}-1-J}+g)-2Q_{h\e,m-2J}(C_J\psi(|s|)|s|^{\frac{\pi}{\alpha}-1-J}+g)\big)\\
&\geq  c_{\e} h^{2J}\big(Q_{h,m-2J}(C_J\psi(|s|)|s|^{\frac{\pi}{\alpha}-1-J})-CQ_{h\e,m-2J}(C_J\psi(|s|)|s|^{\frac{\pi}{\alpha}-1-J})\big)\\
&\qquad -Ch^{2J}(Q_{h,m-2J}(g)+Q_{h\e,m-2J}(g)),
\end{align*}
where for all $n\geq 0$ and $t<\frac{\pi}{\alpha}+\frac{1}{2}$, 
$
|s|^{n-t+J}\partial_s^ng\in L^2$ and 
$$
C_J:=(-1)^J\prod_{j=0}^{J-1}\Big(\frac{\pi}{\alpha}-j-1\Big)\tan \Big((\frac{\pi}{\alpha}-j-1)\frac{\alpha}{2}\Big)\neq 0.
$$
Where the fact that $C_J\neq 0$ follows from $J<\frac{\pi}{\alpha}$ and $0<(\frac{\pi}{\alpha}-j-1)\frac{\alpha}{2}<\pi$. 

Now, by Lemma~\ref{l:highDomain}, since $0\leq \frac{m-2J}{2}= \frac{m}{2}-\lceil\frac{\pi}{\alpha}-\frac{3}{2}\rceil <\frac{\pi}{\alpha}+\frac{1}{2}$, we have
$$
Q_{h,m-2J}(g)\leq Ch^{m-2J}
$$

Now, to finish the first step, observe that 
        \begin{align*}
        &|\xi|^{\frac{\pi}{\alpha}-J}\widehat{(\psi(|\cdot|)|\cdot|^{\frac{\pi}{\alpha}-1-J})}(\pm \xi)\\
        &= \int e^{\mp is} \psi(|\xi|^{-1}|s|)|s|^{\frac{\pi}{\alpha}-1-J}ds\\
        &=\int e^{ is}\psi(s)|s|^{\frac{\pi}{\alpha}-1-J}ds +\int e^{ is} (- D_s)^N \big(\psi(|\xi|^{-1}|s|)-\psi(|s|))|s|^{\frac{\pi}{\alpha}-1-J}\big)ds\\
        &\to \int e^{ is}\psi(s)|s|^{\frac{\pi}{\alpha}-1-J}ds+\int e^{is}( D_s)^N\Big((1-\psi(|s|))|s|^{\frac{\pi}{\alpha}-1-J}\Big)ds= \widehat{|\cdot|^{\frac{\pi}{\alpha}-1-J}}(1)\neq 0,
        \end{align*}
        where the last inequality follows from the fact that $\frac{\pi}{\alpha}\notin \mathbb{Z}$. 
        In particular, 
        \begin{align*}
        &Q_{h,m-2J}(\psi(|s|)|s|^{\frac{\pi}{\alpha}-1-J})\\
        &\simeq \langle h^{m-2J}(I+|D|)^{m-2J}(I+h^{m-2J}(I+|D|)^{m-2J})^{-1}\psi(|s|)|s|^{\frac{\pi}{\alpha}-1-J},\psi(|s|)|s|^{\frac{\pi}{\alpha}-1-J}\rangle \\
        &=\sum |\widehat{\psi(|\cdot|)|\cdot|^{\frac{\pi}{\alpha}-1-J}}(n)|^2 h^{m-2J}(1+|n|)^{m-2J}/(1+h^{m-2J}(1+|n|)^{m-2J})\\
        &= \sum_{|n|\geq M} |n|^{-\frac{2\pi}{\alpha}+2J} h^{m-2J}(1+|n|)^{m-2J}/(1+h^{m-2J}(1+|n|)^{m-2J}) +O(h^{m-2J})\\
        &\simeq h^{\frac{2\pi}{\alpha}-1-2J}
        \end{align*}
        since $m>\frac{2\pi}{\alpha}-1$. This implies the first estimate in~\eqref{e:traceEstimates}. 

Step 3b: the error term

        To obtain the second estimate, we consider operators $R_1$, $R_2,$ $R_3$ with kernels of the form
    \begin{align*}
       K_{R_1}(s,t)&=w_\alpha(t) w_{\alpha}(s)x_\alpha(t)k_1(x_\alpha(t),x_{\alpha}(s))\\
       K_{R_2}(s,t)&=w_\alpha(s) k_2(x_\alpha(t),x_{\alpha}(s)),\\
       K_{R_3}(s,t)&=w_\alpha(t) k_3(x_\alpha(t),x_{\alpha}(s)),\\
        \end{align*}
        with $k_i$ smooth.

        We consider $R_1$, the others being similar.
        To analyze $R_1$, we Taylor expand the kernel to obtain for any $M$
        \begin{align*}
        K_{R_1}(t,s)&=w_\alpha(t)w_{\alpha}(s)x_\alpha(t)\Big(\sum_{j=0,\ell}^{M-1}c_{j,\ell}x_\alpha(t)^jx_\alpha(s)^\ell+\sum_{j=0}^{M-1}x_{\alpha}(s)^Mx_{\alpha}(t)^ja_{j,M}(x_\alpha(s)) \\
        &\qquad +x_\alpha(t)^M\sum_{\ell=0}^{M-1}x_\alpha(s)^\ell b_{\ell,M}
        (x_\alpha(t)) +x_\alpha(s)^M x_\alpha(t)^Me_M(x_\alpha(t),x_\alpha(s))\Big),
        \end{align*}
        with $a_j$, $b_\ell$, and $e$ smooth.

        Hence, for any $M\geq 1$,  
        \begin{multline*}
        R_1= \sum_{j,\ell=0}^{M-1}c_{j,\ell}\Big(\sgn(t)^{j+1}\psi(|t|)|t|^{\frac{(2+j)\pi}{\alpha}-1}\Big)\otimes\Big( \psi(|s|)\sgn(s)^\ell|s|^{\frac{(1+\ell)\pi}{\alpha}-1}\Big)\\
        +\sum_{j=0}^{M-1} \Big(\sgn(t)^{j+1}\psi(|t|)|t|^{\frac{(2+j)\pi}{\alpha}-1}\Big)\otimes g_{j,M} +\sum_{\ell=0}^{M-1}\tilde{g}_{\ell,M}\otimes \Big(\sgn(s)^\ell\psi(|s|)|s|^{\frac{(1+\ell)\pi}{\alpha}-1}\Big)+ R_{1,M}'
        \end{multline*}
        where, with $K_{R_{1,M}'}$, we have for all $n, n'$, that
        \begin{gather*}
        |t|^{-\frac{(M+1)\pi}{\alpha}+\frac{1}{2}+0+n}|s|^{-\frac{(M+1)\pi}{\alpha}+\frac{1}{2}+0+n'}\partial_t^n\partial_s^{n'}K_{R_{1,M}'}\in L^2,
        \\
          |s|^{-\frac{(M+1)\pi}{\alpha}+\frac{1}{2}+0+n}\partial_s^ng_{j,M}\in L^2,\quad |s|^{-\frac{(M+2)\pi}{\alpha}+\frac{1}{2}+0+n}\partial_s^n\tilde{g}_{\ell,M} \in L^2.
         \end{gather*}
        The same analysis as in Step 2, using Lemma~\ref{l:highDomain} then shows that
        $$
        |\tr (T_{h,m}^\theta R_1)|=o(h^{\frac{2\pi}{\alpha}-1})+|\tr (T_{h,m}^\theta R_{1,M}')|
        $$

Now, let $s>1$,   $m_1',m_2'<\min(\frac{\pi}{\alpha}+\frac{1}{2}, \frac{3}{2})$ and $J_1+J_2+m_1'+m_2'\leq m$. Then, using cyclicity of the trace and boundedness of $\Lambda_\theta^{m_i'}:L^2\to H^{-m_i'}$ and $\Lambda_\theta^{m_i'}:H^{m_i'}\to L^2$, we have
\begin{align*}
&|\tr (T_{h,m}^\theta R_{1,M}')|\\
&=|\tr ( (1+(h\Lambda_\theta)^m)^{-(m-J_2-m_2')/m}(h\Lambda_\theta)^{m-m_2'-J_2} R_{1,M}'(h\Lambda_\theta)^{m_2'+J_2}(1+(h\Lambda_\theta)^m)^{-(m_2'+J_2)/m})|\\
&\leq h^{m'+J_1+J_2}\|\Lambda_{\theta}^{m_1'}\Lambda_\theta^{J_1}R_{1,M'}\Lambda_{\theta}^{J_2}\Lambda_{\theta}^{m_2'}\|_{\tr}\\
&\leq h^{m_1'+m_2'+J_1+J_2}\|(1+|D|)^{m_1'}\Lambda_\theta^{J_1}R_{1,M'}\Lambda_{\theta}^{J_2}(1+|D|)^{m_2'}\|_{\tr}\\
&\leq h^{m_1'+m_2'+J_1+J_2}\|(1+|D|)^{m_1'+s}\Lambda_\theta^{J_1}R_{1,M'}\Lambda_{\theta}^{J_2}(1+|D|)^{m_2'}\|_{L^2\to L^2}
\end{align*}
Now, by Lemma~\ref{l:highDomain}, for $m_1'+s<\frac{3}{2}$, $J_1+m_1'+s<\frac{\pi}{\alpha}+\frac{1}{2}$, and $J_2+m_2'<\frac{\pi}{\alpha}+\frac{1}{2}$, we have 
$$
\|(1+|D|)^{m_1'+s}\Lambda_\theta^{J_1}R_{1,M'}\Lambda_{\theta}^{J_2}(1+|D|)^{m_2'}\|_{L^2\to L^2}\leq C,
$$
and hence
$$
|\tr (T_{h,m}^\theta R_{1,M}')|=o(h^{\frac{2\pi}{\alpha}-1}),
$$
which proves the second estimate in~\eqref{e:traceEstimates} and hence the lemma.

\end{proof}

An almost identical proof yields
       \begin{lemma}
       \label{l:traceLower2}
    Let $\alpha\in (0,2\pi)\setminus \mc{E}$ and $m\in \mathbb{R}$ with $\frac{2\pi}{\alpha}-1<m<\frac{2\pi}{\alpha}$. Let $H\Subset \Omega_{\alpha}$  be open with smooth, connected boundary.
Then, $\mathscr{D}^{\flat}(\Omega_{\alpha},H)$ is trace class and there is $c>0$ such that if  $\mathscr{D}^{\flat}(\Omega_{\alpha},H)$ has kernel of the form~\eqref{e:kernelForm1} with $k(0,0)\geq 0$, then
        \begin{gather*}
        \tr\big(T_{h,m}\mathscr{D}^{\flat}(\Omega_{\alpha},H))\geq  c k(0,0)h^{\frac{2\pi}{\alpha}-1}+o(h^{\frac{2\pi}{\alpha}-1}),\\
        T_{h,m}:=\psi_{m}(h(1+\Lambda^{\Omega_{\alpha}})).\nonumber
                \end{gather*}
    \end{lemma}

\section{Proofs of the main theorems}

Before proving our main theorem, we prove a basic estimate on $\sigma_j$. This estimate follows easily from the Weyl law~\cite{Ro:23,KaLaPo:23}, but we choose to present a self-contained proof.
\begin{lemma}\label{l:basic}
    Let $\mc{P}$ be a curvilinear polygon. Then there is $c>0$ such that 
    $$
    cj\leq \sigma_j\leq C^{-1}j.
    $$
\end{lemma}
\begin{proof}
    First, notice that $\sigma_0=0$ and $\sigma_1>0$. Therefore, it is enough to prove that 
    $$
    c(j+1)\leq 1+\sigma_j\leq C(j+1). 
    $$

For a compact operator $A$, let $s_j(A)$ denote the singular values of $A$.  By Lemma~\ref{l:invertible} $(1+\Lambda^{\mc{P}})^{1/2}:L^2\to H^{-1/2}$ is invertible and hence, using~\cite[Proposition B.15]{ZwScat}, we have
\begin{align*}
    (1+cj)^{-1/2}&\leq s_j( (1+|D|)^{-1/2}))\\
    &=s_j((1+|D|)^{-1/2}(1+\Lambda^{\mc{P}})^{1/2}(1+\Lambda^{\mc{P}})^{-1/2})\\
    &\leq Cs_j((1+\Lambda^{\mc{P}})^{-1/2})=C(1+\sigma_j)^{-1/2}
    \end{align*}
    Similarly,
    \begin{align*}(1+\sigma_j)^{-1/2}&=s_j((1+\Lambda^{\mc{P}})^{-1/2})\\
    &=s_j((1+\Lambda^{\mc{P}})^{-1/2}(1+|D|)^{1/2}(1+|D|)^{-1/2})\\
    &\leq Cs_j((1+|D|)^{-1/2})\leq C(1+cj)^{-1/2}.
    \end{align*}
\end{proof}

We can now prove the main theorems. 
\begin{proof}[Proof of Theorem~\ref{t:lowerBound}]
Shrinking $r$ if necessary, we may assume that $B(x_0,r)\Subset \Omega_{\alpha}$. We apply Lemma~\ref{l:basicKernel} to see that, in Lemma~\ref{l:traceLower2} with $H=B(x_0,r/2)$, $k(0,0)>0$ and hence $\frac{2\pi}{\alpha}-1<m<\frac{2\pi}{\alpha}$
$$
\tr T_{h,m}\mathscr{D}^{\flat}(\Omega_{\alpha},H) \geq ch^{\frac{2\pi}{\alpha} -1}. 
$$
Now, let $\{\phi_{\sigma_j}\}_{j=0}^\infty\subset L^2(\partial\Omega_{\alpha})$ be an orthonormal basis of eigenfunctions for $\Lambda^{\Omega_{\alpha}}$ with eigenvalue $\sigma_j$ and $u_{\sigma_j}$ the corresponding Harmonic extensions. Then, let $\chi \in C_c^\infty(B(x_0,r))$ with $\supp (1-\chi)\cap B(x_0,r/2)=\emptyset$. Then, by Lemma~\ref{l:differenceDescribe}, 
\begin{align*}
ch^{\frac{2\pi}{\alpha}-1}\leq \tr T_{h,m}\mathscr{D}^{\flat}(\Omega_{\alpha},H)&=\sum_j \langle  T_{h,m}\mathscr{D}^{\flat}(\Omega_{\alpha},H))\phi_j,\phi_j\rangle \\
&=\sum_j \frac{h^m(1+\sigma_j)^m}{1+h^m(1+\sigma_j)^m}\langle  \mathscr{D}^{\flat}(\Omega_{\alpha},H)\phi_j,\phi_j\rangle\\
&\leq \sum_j \frac{h^m(1+\sigma_j)^m}{1+h^m(1+\sigma_j)^m}\int_{\Omega_{\alpha}}|\nabla \chi u_{\sigma_j}|^2dx\\
&\leq \sum_j C\frac{h^m(1+\sigma_j)^m}{1+h^m(1+\sigma_j)^m}\|u_{\sigma_j}\|_{L^2(B(x_0,r))}^2,
\end{align*}
where we use that $u_{\sigma_j}$ is harmonic in the last line. 

Suppose now that $\|u_{\sigma_j}\|_{L^2(B(x_0,r))}=o(j^{-\frac{\pi}{\alpha}})$. Since by Lemma~\ref{l:basic} there is $c>0$ such that $cj\leq \sigma_j\leq Cj$, for any $\e>0$, we have
$$
ch^{\frac{2\pi}{\alpha}-1}\leq \sum_j C\frac{h^m(1+j)^m}{1+h^m(1+j)^m}o(j^{-\frac{2\pi}{\alpha}})\leq C_\e h^m  +C\e \sum_j\frac{h^m(1+j)^m}{1+h^m(1+j)^m}j^{-\frac{2\pi}{\alpha}}\leq Ch^m +C\e h^{\frac{2\pi}{\alpha}-1},
$$
which is a contradiction for $\e>0$ and $h$ small enough.
\end{proof}

     \begin{proof}[Proof of Theorem~\ref{t:gap}]
 
         Let $x_0\in \Omega_{\alpha}$ with $d(x_0,\partial\Omega_{\alpha})>r_0$ and $y_\e\in\Omega_{\alpha}$ $|y_\e|> r_0$ and 
         $$
         \e< d(y_0,\partial\Omega_{\alpha})<2\e. 
         $$
         
         Let $H_1=B(x_0,\e)$ and $H_0=B(y_0,\e)$. Then, by Lemmas~\ref{l:closeToBoundary},~\ref{l:kernelLower} (with $p>4$), and~\ref{l:kernelForm2}, for $\e$ small enough, $\mathscr{D}(H_1,H_0)$ has a kernel of the form~\eqref{e:kernelForm} with $K_{\Omega\Omega}(0,0)>0$. Lemma~\ref{l:traceLower} then implies that  for $\frac{2\pi}{\alpha}-1<m<\frac{2\pi}{\alpha}$
\begin{equation} \label{e:lowBoundToUse}
\inf_{\theta\in[0,1]}\tr T_{h,m}^\theta \mathscr{D}(H_1,H_0) \geq ch^{\frac{2\pi}{\alpha} -1}. 
\end{equation}

Define
$$
\Phi_{m}(x):=\int_0^x \psi_{m}(t)dt.
$$
To prove Theorem~\ref{t:gap}, we first show that 
\begin{equation} \label{e:toProveTrace}
``\tr''(\Phi_m(h\Lambda_1)-\Phi_m(h\Lambda_0))\geq \int_0^1\tr T_{h,m}^\theta h\mathscr{D}(H_1,H_0)d\theta \geq ch^{\frac{2\pi}{\alpha}},
\end{equation}
where $``\tr''$ is an appropriately regularized trace, $\Lambda_1:=1+\Lambda^{\Omega_{\alpha}\setminus H_1}$, and $\Lambda_0:=1+\Lambda^{\Omega_{\alpha}\setminus H_0}$.

To dot his, let $q_n$ by polynomials with $q_n\to \Phi_m$ in $C^1$ uniformly on compact sets. Let $\{e_j\}_{j=0}^\infty\subset L^2(\mathbb{S}^1)\times  L^2(\mathbb{S}^1)$ be an orthonormal basis of eigenvectors of $\Lambda_1$ with eigenvalues $1+\sigma_j(\Lambda^{\Omega_\alpha\setminus H_1})$. Let $\Pi_M$ denote the orthogonal projector onto $\{e_j\}_{j=0}^M$. Then, by cyclicity of the trace
 \begin{align*}
 \tr  \big( \Pi_M\big( q_n(\Pi_Mh\Lambda_1\Pi_M)-q_n(\Pi_Mh\Lambda_0\Pi_M)\big)\Pi_M\big)&=\int_0^1 \tr( q_n'(\Pi_M h\Lambda_\theta\Pi_M)\Pi_M h\mathscr{D}(H_1,H_0)\Pi_M )d\theta\\
 &\underset{n\to \infty}{\longrightarrow} \int_0^1 \tr  \psi_m(\Pi_M h\Lambda_\theta\Pi_M )\Pi_M h \mathscr{D}(H_1,H_0)\Pi_Md\theta.
 \end{align*}
 
 On the other hand 
 $$
 \tr \Pi_Mq_n(\Pi_Mh\Lambda_j\Pi_M)\Pi_M\to \tr \Pi_M\Phi_m(\Pi_M h\Lambda_j \Pi_M)\Pi_M,
 $$
 so
 $$
 \tr \big(\Pi_M\big(\Phi_m(\Pi_M h\Lambda_1 \Pi_M)-\Phi_m(\Pi_M h\Lambda_0\Pi_M)\big)\Pi_M\big)=\int_0^1\tr \psi_m(\Pi_Mh\Lambda_\theta\Pi_M)\Pi_M h\mathscr{D}(H_1,H_0)\Pi_M d\theta.
 $$

 Since $\mathscr{D}(H_1,H_0)$ is trace class, 
 \begin{equation}
 \label{e:trace1}
 \| \Pi_M h\mathscr{D}(H_1,H_0)\Pi_M- h\mathscr{D}(H_1,H_0)\|_{\tr}\underset{M\to \infty}{\longrightarrow} 0.
 \end{equation}

Next, observe that
\begin{equation}
\label{e:trace2}
\psi_m(\Pi_Mh\Lambda_\theta\Pi_M)=I -(I+(\Pi_M h\Lambda_\theta\Pi_M)^m)^{-1}.
\end{equation}
Then, since $h\Lambda_\theta \geq h$, letting $\Gamma= \{\frac{h}{2}-it+m^2|t|\,:\, t\in\mathbb{R}\}$ be a contour around $[h,\infty)$, we have 
\begin{align*}
&(I+(\Pi_M h\Lambda_\theta\Pi_M)^m)^{-1}-(I+(h\Lambda_\theta)^m)^{-1}\\
&= \frac{h}{2\pi i}\int_{\Gamma} \frac{1}{1+z^m} (h\Lambda_\theta-z)^{-1}(\Pi_M\Lambda_\theta \Pi_M-\Lambda_\theta)(\Pi_Mh\Lambda_\theta \Pi_M -z)^{-1}dz+(I-\Pi_M)\\
&= -\frac{h}{2\pi i}\int_{\Gamma} \frac{1}{1+z^m} 
(h\Lambda_\theta-z)^{-1}(I-\Pi_M)\Lambda_1(I-\Pi_M)(\Pi_Mh\Lambda_\theta \Pi_M -z)^{-1}dz \\
&\qquad -(1-\theta)\frac{h}{2\pi i}\int_{\Gamma} \frac{1}{1+z^m} 
(h\Lambda_\theta-z)^{-1} (\Pi_M\mathscr{D}(H_1,H_0)\Pi_M-\mathscr{D}(H_1,H_0))(\Pi_Mh\Lambda_\theta \Pi_M -z)^{-1}dz\\
&\qquad +(I-\Pi_M)\\
&= -\frac{h}{2\pi i}\int_{\Gamma} \frac{1}{1+z^m} (h\Lambda_\theta-z)^{-1}(h\Lambda_1-z)(h\Lambda_1-z)^{-1}(I-\Pi_M)\Lambda_1(I-\Pi_M)(\Pi_Mh\Lambda_\theta \Pi_M -z)^{-1}dz\\
&\qquad -\frac{h}{2\pi i}(1-\theta)  \int_{\Gamma} \frac{1}{1+z^m} (h\Lambda_\theta-z)^{-1}(\Pi_M\mathscr{D}(H_1,H_0)\Pi_M-\mathscr{D}(H_1,H_0))(\Pi_Mh\Lambda_\theta \Pi_M -z)^{-1}dz\\
&\qquad +(I-\Pi_M)\\
&=\frac{1}{2\pi i} \int_{\Gamma} \frac{1}{1+z^m} (h\Lambda_\theta-z)^{-1}(h\Lambda_1-z)(I+z(h\Lambda_1-z)^{-1})z^{-1}(I-\Pi_M)dz+(I-\Pi_M)\\
&\qquad -\frac{h}{2\pi i}(1-\theta)  \int_{\Gamma} \frac{1}{1+z^m} (h\Lambda_\theta-z)^{-1}(\Pi_M\mathscr{D}(H_1,H_0)\Pi_M-\mathscr{D}(H_1,H_0))(\Pi_Mh\Lambda_\theta \Pi_M -z)^{-1}dz\\
&=:A(I-\Pi_M)+B_M.
\end{align*}
$B_M$ converges to zero since $\mathscr{D}(H_1,H_0)$ is trace class, $m>0$, and the resolvents are bounded by $d(z,[h,\infty))^{-1}$. On the other hand $A$ is bounded since $(h\Lambda_\theta-z)^{-1}(h\Lambda_1-z)$ is bounded and $(h\Lambda-z)^{-1}\leq d(z,[h,\infty))^{-1}$. 

In particular, since $\|(I-\Pi_M)\mathscr{D}(H_1,H_0)\|_{\tr}\to 0$, we have
\begin{equation} \label{e:trace3}
\|\big((I+(\Pi_M h\Lambda_\theta\Pi_M)^m)^{-1}-(I+(h\Lambda_\theta)^m)^{-1}\big)\mathscr{D}(H_1,H_0)\|_{\tr}\underset{M\to \infty}{\longrightarrow} 0
\end{equation}
and hence by~\eqref{e:lowBoundToUse}, \eqref{e:trace1}, \eqref{e:trace2}, and~\eqref{e:trace3}, we have
\begin{equation}
\label{e:toContradict}
\lim_{M\to \infty} \tr\big(\Phi_m(\Pi_Mh\Lambda_1\Pi_M)-\Phi_m(\Pi_Mh\Lambda_0\Pi_M)\big)=\int_0^1\tr T_{h,m}^\theta h\mathscr{D}(H_1,H_0)d\theta\geq ch^{\frac{2\pi}{\alpha}}.
\end{equation}

Now, observe that, since $\Phi_m$ is increasing and 
$$
\sigma_j(h\Lambda_0)=\min_{\dim V=j+1}\max_{v\in V,\, \|v\|=1} \langle h\Lambda_0v,v\rangle ,
$$
we have 
$$
\sum_{j=0}^M\Phi_m(\sigma_j(h\Lambda_0))\leq \tr\big(\Phi_m(\sigma_j(\Pi_M h\Lambda_0\Pi_M))\big),
$$
and hence
\begin{multline}\label{e:trace4}
ch^{\frac{2\pi}{\alpha}}\leq \lim_{M\to \infty} \tr\big(\Phi_m(\Pi_Mh\Lambda_1\Pi_M)-\Phi_m(\Pi_Mh\Lambda_0\Pi_M)\big)\\
\leq \liminf_{M\to \infty}\sum_{j=0}^M \Phi_m(h(1+\sigma_j(\Lambda^{\Omega_\alpha\setminus H_1})))-\Phi_m(h(1+\sigma_j(\Lambda^{\Omega_\alpha\setminus H_0}))).
\end{multline}

Now that we have proved~\eqref{e:toProveTrace} in the sense of~\eqref{e:trace4} suppose by contradiction that 
$$
|\sigma_j(\Lambda^{\Omega_{\alpha}\setminus  H_1})-\sigma_j(\Lambda^{\Omega_{\alpha}\setminus H_0})|=o(j^{-\frac{2\pi}{\alpha}}). 
$$
Then, since by Lemma~\ref{l:basic} $cj\leq \sigma_j(\Lambda^{\Omega_\alpha\setminus H_1})\leq Cj$ and $cj\leq \sigma_j(\Lambda^{\Omega_\alpha\setminus H_0})\leq Cj$, for all $\e>0$ we have
\begin{align*}
ch^{\frac{2\pi}{\alpha}}&\leq \liminf_{M\to\infty} \sum_{j=0}^M\Phi(h(1+\sigma_j(\Lambda^{\Omega_{\alpha}\setminus H_1})))-\Phi(h(1+\sigma_j(\Lambda^{\Omega_{\alpha}\setminus H_0})) \\
&\leq h\liminf_{M\to \infty}\sum_{j=1}^M \sup_{x\in h[1+cj,1+Cj]}|\psi_m(x)||\sigma_j(\Lambda^{\Omega_{\alpha}\setminus H_1})-\sigma_j(\Lambda^{\Omega_{\alpha}\setminus H_0})|\\
&\leq Ch\sum_{j}\min ( h^m(1+j)^m, 1)o(j^{-\frac{2\pi}{\alpha}})\\
&=C_\e h^{m+1}+C\e h \sum_j \min (h^m(1+j)^m,1)j^{-\frac{2\pi}{\alpha}}\leq C_\e h^{m+1}+C\e h^{\frac{2\pi}{\alpha}}.
\end{align*}
Taking $\e$ small enough and then $h$ small enough, we obtain a contradiction with~\eqref{e:trace4}. 
    \end{proof}

\appendix 

    \section{Domains of the DtN map on Curvilinear polygons}
In this section, we study the DtN map on curvilinear polygons. In particular, we show that they satisfy the assumptions of Lemma~\ref{l:traceLower}.
\begin{lemma}
    \label{l:bounded}
    Let $\mc{P}\subset \mathbb{R}^2$ be a curvilinear polygon with maximal angle $\alpha$. Then for all $\max(-\frac{\pi}{\alpha}+\frac{1}{2},-\frac{1}{2})< s<\min(\frac{\pi}{\alpha}+\frac{1}{2},\frac{3}{2})$, $\Lambda^{\mc{P}}:H^s(\partial\mc{P})\to H^{s-1}(\partial\mc{P})$. 
\end{lemma}
\begin{proof}
We prove the statement for $\frac{1}{2}\leq s<\min(\frac{\pi}{\alpha}+\frac{1}{2},\frac{3}{2})$ with the other indeces following by duality.

     Let $1\leq t<\frac{\pi}{\alpha}+1$.  We recall some elliptic regularity results from~\cite[Theorem 23.3]{Da:88a}.
     Then 
    for all $u\in H_0^1(\mc{P})$, $f\in H^{t-2}(\mc{P})$ such that
    \begin{equation} \label{e:Dirichlet}
    \langle \nabla u,\nabla v\rangle =\langle f,v\rangle\text{ for all }v\in H_0^1(\mc{P}),
    \end{equation}
    we have $u\in H^t(\mc{P})$ and 
    $$
    \|u\|_{H^t(\mc{P})}\leq C(\|f\|_{H^{t-2}(\mc{P})}+\|u\|_{H^1(\mc{P})}).
    $$
In particular, since~\eqref{e:Dirichlet} together with the Poincar\'e inequality implies
$$
\|u\|_{H^1(\mc{P})}^2\leq C\|\nabla u\|_{L^2(\mc{P})}^2=\langle f,u\rangle, 
$$
we have for all $u\in H_0^1(\mc{P})$ satisfying~\eqref{e:Dirichlet},
\begin{equation}
\label{e:DirichletFinal}
    \|u\|_{H^t(\mc{P})}\leq C\|f\|_{H^{t-2}(\mc{P})},
\end{equation}
which implies that for all $f\in H^{t-2}$ there is a unique solution $u\in H^t(\mc{P})\cap H_0^1(\mc{P})$ of~\eqref{e:Dirichlet}.

Let $1< s<\min(\frac{\pi}{\alpha}+\frac{1}{2},\frac{3}{2})$. Then, by~\cite[Theorem 1.5.2.8]{grisvard2011elliptic}, there is $E:H^{s}(\partial\mc{P})\to H^{s+\frac{1}{2}}(\mc{P})$ satisfying
    $$
    (Ef)|_{\partial\mc{P}}=f
    $$
    and $\|E\|_{H^{s-\frac{1}{2}}(\partial\mc{P})\to H^s(\mc{P})}<\infty$.
    
    Next, let $w\in H_0^1(\mc{P})$ satisfy
    $$
    \langle \nabla w,\nabla v\rangle =\langle \nabla Ef,\nabla v\rangle\text{ for all }v\in H_0^1(\mc{P}).
    $$
    Then, 
    $$
    \|\Lambda^{\mc{P}} f\|_{H^{s-1}(\partial\mc{P})}=\|\partial_\nu (Ef-w)\|_{H^{s-1}(\partial\mc{P})}\leq C(\|w\|_{H^{s+\frac{1}{2}}(\mc{P})}+\|Ef\|_{H^{s+\frac{1}{2}(\mc{P})}})\leq C\|Ef\|_{H^{s+\frac{1}{2}}(\mc{P})}\leq C\|f\|_{H^{s}(\partial\mc{P})}.
    $$
    In particular, 
    $$
    \|\Lambda^{\mc{P}}\|_{H^{s}(\partial\mc{P})\to H^{s-1}(
\partial\mc{P})}<\infty.
    $$
    The estimates for $\frac{1}{2}\leq s\leq 1$ follow by interpolation with the energy bound $\Lambda^{\mc{P}}:H^{\frac{1}{2}}(\partial\mc{P})\to H^{-\frac{1}{2}}(\partial\mc{P})$. 
\end{proof}

We will use the following analog of the Poincar\'e inequality
\begin{lemma}
       There is $C>0$ such that for all $u\in H^1(\mc{P})$, 
    \begin{equation} \label{e:poincare}
    \|u\|_{H^1(\mc{P})}\leq C(\|\nabla u\|_{L^2(\mc{P})}+\|u|_{\partial\mc{P}}\|_{L^2(\partial\mc{P})}).
    \end{equation}
\end{lemma}
\begin{proof}
        We prove~\eqref{e:poincare} by contradiction. Suppose~\eqref{e:poincare} fails. Then there are $u_n\in H^1(\mc{P})$ with $\|u_n\|_{H^1(\mc{P})}=1$, $\|\nabla u_n\|_{L^2(\mc{P})}\to 0$ and $\|u_n|_{\partial\mc{P}}\|_{L^2(\partial\mc{P})}\to 0$. In particular, extracting a subsequence, we may assume that  $u_n\rightharpoonup u$ in $H^1(\mc{P})$ and $u_n\to u$ in $H^{\frac{3}{4}}(\mc{P})$. First, notice that since the trace map is bounded from $H^{3/4}(\mc{P})\to L^2(\partial\mc{P})$, we have $u|_{\partial\mc{P}}=0$. Next, since $\|\nabla u_n\|_{L^2(\mc{P})}\to 0$, we have $\nabla u=0$ so that $u$ is a constant.  Thus, since $u|_{\partial\mc{P})}=0$, $u=0$. But then $\|u_n\|_{L^2}\to 0$ and which is a contradiction since
    $$
    1=\|u_n\|^2_{H^1(\mc{P})}=\|\nabla u_n\|_{L^2(\mc{P})}^2+\|u_n\|_{L^2(\mc{P})}^2\to 0. 
    $$
\end{proof}

\begin{lemma}
    \label{l:invertible}
    Let $\mc{P}\subset \mathbb{R}^2$ be a curvilinear polygon with maximal angle $\alpha$. Then for all $\max(-\frac{\pi}{\alpha}+\frac{1}{2},-\frac{1}{2})< s<\min(\frac{\pi}{\alpha}+\frac{1}{2},\frac{3}{2})$, $(1+\Lambda)^{-1}:H^{s-1}(\partial\mc{P})\to H^s(\partial\mc{P})$
\end{lemma}
\begin{proof}
We consider the $\frac{1}{2}\leq s<\min(\frac{\pi}{\alpha}+\frac{1}{2},\frac{3}{2})$. The other values of $s$ follow by duality.

\noindent{Step 1:} Invertibility in $H^{\frac{1}{2}}$. 

    First observe that for $g\in H^{1/2}(\partial\mc{P})$, 
    $$
    \langle (1+\Lambda^{\mc{P}}) g,g\rangle_{L^2(\partial\mc{P})}= \|\nabla \mc{H}^{\mc{P}}g\|_{L^2(\mc{P})}^2+\|g\|_{L^2(\partial\mc{P})}^2,
    $$
    where 
    $$
    -\Delta \mc{H}^{\mc{P}}g=0\text{ in }\mc{P},\qquad \mc{H}g|_{\partial\mc{P}}=g.
    $$
    In particular, by~\eqref{e:poincare} and the fact that the trace is bounded from $H^1(\mc{P})\to H^{\frac{1}{2}}(\partial\mc{P})$, 
    $$
    \|g\|_{H^{\frac{1}{2}}(\partial\mc{P})}^2\leq C\|\mc{H}^{\mc{P}}g\|_{H^1(\mc{P})}^2\leq C(\|\nabla \mc{H}^{\mc{P}}g\|_{L^2(\mc{P})}^2+\|g\|_{L^2(\mc{P})}^2)=\langle (1+\Lambda^{\mc{P}})g,g\rangle.
    $$
    Hence, 
    $$
    \|g\|_{H^{\frac{1}{2}}(\partial\mc{P})}\leq C\|(1+\Lambda^{\mc{P}})g\|_{H^{-\frac{1}{2}}(\partial\mc{P})}
    $$
    and, since $\Lambda^{\mc{P}}:H^{\frac{1}{2}}\to H^{-\frac{1}{2}}$ is self-adjoint, this implies $(1+\Lambda^{\mc{P}}):H^{\frac{1}{2}}(\partial\mc{P})\to H^{-\frac{1}{2}}(\partial\mc{P})$ is invertible. 

\noindent Step 2:  Elliptic regularity for the DtN map at higher regularity

Let $\frac{1}{2}\leq s< \frac{\pi}{\alpha}+\frac{1}{2}$. Then, by~\cite[Theorem 23.3]{Da:88a}, for all $w\in H^1(\mc{P})$ satisfying
\begin{equation} \label{e:neumann}
\langle \nabla w,\nabla v\rangle = \langle f, v\rangle\qquad \text{for all }v\in H^1(\mc{P}),
\end{equation}
with $f\in \dot H^{-1}\cap  H^{s-\frac{3}{2}}(\mc{P})$,
we have $w\in H^{s+\frac{1}{2}}(\mc{P})$ and 
    $$
    \|w\|_{H^{s+\frac{1}{2}}(\mc{P})}\leq C(\|f\|_{H^{s-\frac{3}{2}}(\mc{P})}+\|w\|_{H^1(\mc{P)}}).
    $$
     Consider now the space $\underline{H}^s(\mc{P}):=\{ u\in H^s(\mc{P})\,:\,\int u=0\}$. Then, the same arguments leading to~\eqref{e:DirichletFinal} but applying the Poincar\'e inequality in $\underline{H}^1$ shows that for all $f\in \dot {\underline{H}}^{-1}\cap H^{s-\frac{3}{2}}(\mc{P})$, there is $w\in \underline{H}^s(\mc{P})$ satisfying~\eqref{e:neumann} and
    $$
    \|w\|_{H^{s+\frac{1}{2}}(\mc{P})}\leq C\|f\|_{H^{s-\frac{3}{2}}(\mc{P})}.
    $$

    Now, let $1< s<\min(\frac{\pi}{\alpha}+\frac{1}{2},\frac{3}{2})$. Then, by~\cite[Theorem 1.5.2.8]{grisvard2011elliptic}, there is $E_\nu:H^{s-1}(\partial\mc{P})\to H^{s+\frac{1}{2}}(\mc{P})$ such that 
    $$
    (\partial_\nu E_\nu)|_{\partial\mc{P}} =I,\qquad E_\nu|_{\partial\mc{P}}=0
    $$

    Suppose that $g\in H^{\frac{1}{2}}(\partial\mc{P})$ with $\Lambda^{\mc{P}}g=:\tilde{f}\in H^{s-1}(\partial\mc{P})$. Note that 
    $$
     \langle \tilde{f},1\rangle_{\partial\mc{P}}=\langle \Lambda^{\mc{P}}g,1\rangle = \langle \nabla \mc{H}^{\mc{P}}g,\nabla 1\rangle =0.
    $$
    Therefore, there is $w\in H^1(\mc{P})$ such that 
    $$
    \langle \nabla w,\nabla v\rangle_{L^2(\mc{P})} =\langle \nabla E_\nu \tilde{f},\nabla v\rangle_{\mc{P}} -\langle \tilde{f},v\rangle_{\partial\mc{P}},\qquad \text{for all }v\in H^1(\mc{P}),
    $$
    and 
    $$
    \|w\|_{H^{s+\frac{1}{2}}(\mc{P})}\leq C(\|\nabla E_{\nu}\tilde{f}\|_{H^{s-\frac{1}{2}}(\mc{P})}+\|\tilde{f}\|_{H^{s-1}(\partial\mc{P})})\leq C\|\tilde{f}\|_{H^{s-1}(\partial\mc{P})}.
    $$

    Now, observe that $u:=\mc{H}^{\mc{P}}g+w$ satisfies
    $$
    \langle \nabla u,\nabla v\rangle = \langle \nabla E_\nu \tilde{f},\nabla v\rangle_{\mc{P}}\qquad \text{for all }v\in H^1(\mc{P}).
    $$
    Therefore, 
    \begin{align*}
    \|u\|_{H^{s+\frac{1}{2}}(\mc{P})}&\leq C(\|\nabla E_\nu \tilde{f}\|_{H^{s-\frac{1}{2}}}+\|u\|_{H^1(\mc{P})})\\
    &\leq C\|\tilde{f}\|_{H^{s-1}(\partial\mc{P})} +\|\mc{H}^{\mc{P}}g\|_{H^1(\mc{P})}\leq C(\|\tilde{f}\|_{H^{s-1}(\mc{P})} +\|g\|_{H^{\frac{1}{2}}(\partial\mc{P})}).
    \end{align*}
    Hence, 
    \begin{equation}\label{e:dtNElliptic}
    \|g\|_{H^{s}(\partial\mc{P})}\leq C\|\mc{H}^{\mc{P}}g\|_{H^{s+\frac{1}{2}}(\mc{P})}\leq C(\|u\|_{H^{s+\frac{1}{2}}(\partial\mc{P})}+\|w\|_{H^{s+\frac{1}{2}}(\partial\mc{P})})\leq C(\|\tilde{f}\|_{H^{s-1}(\partial\mc{P})}+\|g\|_{H^{\frac{1}{2}}(\partial\mc{P})}).
        \end{equation}

    Next, suppose that $(1+\Lambda^{\mc{P}})g=f$. Then, setting $\tilde{f}=\Lambda^{\mc{P}}g=f-g$ in~\eqref{e:dtNElliptic}, we obtain
    \begin{equation}
    \label{e:onePlusElliptic}
    \|g\|_{H^{s}(\partial\mc{P})}\leq C(\|f\|_{H^{s-1}(\partial\mc{P})}+\|g\|_{H^{s-1}(\partial\mc{P})}+\|g\|_{H^{\frac{1}{2}}(\partial \mc{P})})\leq C\|f\|_{H^{s-1}(\partial\mc{P})},
    \end{equation}
    where we have used that, since $s<\frac{3}{2}$, by Step 1
    $$
    \|g\|_{H^{s-1}(\partial\mc{P})}+\|g\|_{H^{\frac{1}{2}}(\partial\mc{P})}\leq \|(1+\Lambda^{\mc{P}})g\|_{H^{-\frac{1}{2}(\partial\mc{P})}}\leq \|f\|_{H^{-\frac{1}{2}}(\partial\mc{P})},
    $$
    to absorb the $H^{1/2}$ and $H^{s-1}$ errors.
\end{proof}
The following is now a direct corollary of interpolation together with Lemmas~\ref{l:bounded} and~\ref{l:invertible}
\begin{corollary}
    \label{c:domains}
    Let $\mc{P}\subset \mathbb{R}^2$ be a curvilinear polygon with maximal angle $\alpha$. Then for $m\in\mathbb{R}$ with  $0\leq m<\min(\frac{2\pi}{\alpha}+1,3)$, there is $C>0$ such that 
    $$
    \|(1+\Lambda^{\mc{P}})^m\|_{H^{\frac{m}{2}}\to H^{-\frac{m}{2}}}+\|(1+\Lambda^{\mc{P}})^{-m}\|_{ H^{-\frac{m}{2}}\to H^{\frac{m}{2}}}\leq C.
    $$
\end{corollary}

\bibliography{references}

@article{ArMa:12,
  author = {Arendt, Wolfgang and Mazzeo, Rafe},
  title = {Friedlander's eigenvalue inequalities and the {D}irichlet-to-{N}eumann semigroup},
  journal = {Commun. Pure Appl. Anal.},
  fjournal = {Communications on Pure and Applied Analysis},
  volume = {11},
  year = {2012},
  number = {6},
  pages = {2201--2212},
  issn = {1534-0392,1553-5258},
  mrclass = {35J05 (35P15)},
  mrnumber = {2912743},
  mrreviewer = {Davide\ Zucco},
  doi = {10.3934/cpaa.2012.11.2201},
  url = {https://doi.org/10.3934/cpaa.2012.11.2201},
}

@article{DeLaPa:26,
author={Deleporte, Alix and Lagac\'e-Taylor, Jean and Parnovski, Leonid},
title={Listening to boundary regularity in the {S}teklov problem:
Gaps, shifts, and eigenfunction decay},
year={2026},
journal={in preparation},
}

@article{HLu,
  author = {Hislop, P. D. and Lutzer, C. V.},
  coden = {INPEEY},
  doi = {10.1088/0266-5611/17/6/313},
  fjournal = {Inverse Problems. An International Journal on the Theory and Practice of Inverse Problems, Inverse Methods and Computerized Inversion of Data},
  issn = {0266-5611},
  journal = {Inverse Problems},
  mrclass = {35P20 (35R30 47F05)},
  mrnumber = {1872919},
  mrreviewer = {Georgi E. Karadzhov},
  number = {6},
  pages = {1717--1741},
  title = {Spectral asymptotics of the {D}irichlet-to-{N}eumann map on multiply connected domains in {$\Bbb R^d$}},
  url = {http://dx.doi.org/10.1088/0266-5611/17/6/313},
  volume = {17},
  year = {2001},
}

@article{GalkToth,
  author = {Galkowski, Jeffrey and Toth, John A.},
  title = {Pointwise bounds for {S}teklov eigenfunctions},
  journal = {J. Geom. Anal.},
  fjournal = {Journal of Geometric Analysis},
  volume = {29},
  year = {2019},
  number = {1},
  pages = {142--193},
  issn = {1050-6926},
  mrclass = {58J50 (35P20 35S05)},
  mrnumber = {3897008},
  mrreviewer = {Luigi Provenzano},
  doi = {10.1007/s12220-018-9984-7},
  url = {https://doi.org/10.1007/s12220-018-9984-7},
}

@article{GPPS,
  author = {Girouard, Alexandre and Parnovski, Leonid and Polterovich, Iosif and Sher, David A.},
  doi = {10.1017/S030500411400036X},
  fjournal = {Mathematical Proceedings of the Cambridge Philosophical Society},
  issn = {0305-0041},
  journal = {Math. Proc. Cambridge Philos. Soc.},
  mrclass = {58J53 (35J05 35P20 35R01 58J37 58J40 58J50)},
  mrnumber = {3286514},
  mrreviewer = {Jie Yang},
  number = {3},
  pages = {379--389},
  title = {The {S}teklov spectrum of surfaces: asymptotics and invariants},
  url = {http://dx.doi.org/10.1017/S030500411400036X},
  volume = {157},
  year = {2014},
}

@article{LPPS2022,
  title = {Sloshing, Steklov and corners: Asymptotics of Steklov eigenvalues for curvilinear polygons},
  author = {Levitin, Michael and Parnovski, Leonid and Polterovich, Iosif and Sher, David A},
  journal = {Proceedings of the London Mathematical Society},
  volume = {125},
  number = {3},
  pages = {359--487},
  year = {2022},
  publisher = {Wiley Online Library},
}

@article{GaMaWa:26,
  title = {Scattering for the {D}irichlet to {N}eumann map and the {S}teklov spectrum},
  author = {Galkowski, Jeffrey and Malagutti, Marcello and Wang, Ruoyu},
  year = {2026},
  journal = {preprint},
}

@book{Da:88a,
  title = {Elliptic boundary value problems on corner domains},
  author = {Dauge, M.},
  series = {Lecture Notes in Mathematics},
  number = {1341},
  publisher = {Springer-Verlag},
  year = {1988},
}

@book{grisvard2011elliptic,
  title = {Elliptic problems in nonsmooth domains},
  author = {Grisvard, Pierre},
  year = {2011},
  publisher = {SIAM},
}

@article{Ro:23,
  title = {Weyl asymptotics for Poincar{\'e}--Steklov eigenvalues in a domain with Lipschitz boundary},
  author = {Rozenblum, Grigori},
  journal = {Journal of Spectral Theory},
  volume = {13},
  number = {3},
  pages = {755},
  year = {2023},
  publisher = {European Mathematical Society (EMS)},
}

@article{KaLaPo:23,
  title = {Weyl’s Law for the Steklov Problem on Surfaces with Rough Boundary: M. Karpukhin et al.},
  author = {Karpukhin, Mikhail and Lagac{\'e}, Jean and Polterovich, Iosif},
  journal = {Archive for Rational Mechanics and Analysis},
  volume = {247},
  number = {5},
  pages = {77},
  year = {2023},
  publisher = {Springer},
}

@book{ZwScat,
  author = {S. Dyatlov and M. Zworski},
  title = {Mathematical theory of scattering resonances},
  series = {Graduate Studies in Mathematics},
  volume = {200},
  publisher = {American Mathematical Society, Providence, RI},
  year = {2019},
}

@article{BaNiZi:05,
  author = {Constantin Bacuta and Victor Nistor and Ludmil T. Zikatanov},
  title = {Improving the Rate of Convergence of High-Order Finite Elements on Polyhedra I: A Priori Estimates},
  journal = {Numerical Functional Analysis and Optimization},
  volume = {26},
  number = {6},
  pages = {613--639},
  year = {2005},
  publisher = {Taylor \& Francis},
  doi = {10.1080/01630560500377295},
  url = {  https://doi.org/10.1080/01630560500377295 },
  eprint = {  https://doi.org/10.1080/01630560500377295 },
}

@article{abatangelo2019spectral,
  title = {Spectral stability under removal of small capacity sets and applications to Aharonov-Bohm operators},
  author = {Abatangelo, Laura and Felli, Veronica and Hillairet, Luc and L{\'e}na, Corentin and others},
  journal = {J. Spectr. Theory},
  volume = {9},
  number = {2},
  pages = {379--427},
  year = {2019},
}

@book{brezis2011functional,
  title = {Functional analysis, Sobolev spaces and partial differential equations},
  author = {Brezis, Haim and Br{\'e}zis, Haim},
  volume = {2},
  number = {3},
  year = {2011},
  publisher = {Springer},
}

@article{davies2013inverse,
  title = {An inverse spectral theorem},
  author = {Davies, Edward Brian},
  journal = {Journal of Operator Theory},
  pages = {195--208},
  year = {2013},
  publisher = {JSTOR},
}

@inproceedings{smith2016heat,
  title = {Heat traces and existence of scattering resonances for bounded potentials},
  author = {Smith, Hart F and Zworski, Maciej},
  booktitle = {Annales de l'Institut Fourier},
  volume = {66},
  number = {2},
  pages = {455--475},
  year = {2016},
}

@inproceedings{christiansen2017sharp,
  title = {A sharp lower bound for a resonance-counting function in even dimensions},
  author = {Christiansen, TJ},
  booktitle = {Annales de l'Institut Fourier},
  volume = {67},
  number = {2},
  pages = {579--604},
  year = {2017},
}

@incollection{molchanov2009general,
  title = {On general Cwikel--Lieb--Rozenblum and Lieb--Thirring inequalities},
  author = {Molchanov, Stanislav and Vainberg, Boris},
  booktitle = {Around the Research of Vladimir Maz'ya III: Analysis and Applications},
  pages = {201--246},
  year = {2009},
  publisher = {Springer},
}
 \bibliographystyle{alpha}
\end{document}